\documentclass[12pt, a4paper, parskip=half, abstracton]{scrartcl}

\usepackage{array}
\usepackage{marginnote}
\usepackage{xcolor}
\usepackage{amscd,amssymb,amsfonts,amsmath,latexsym,amsthm}
\usepackage{hyperref}
\usepackage[all,cmtip]{xy}
\usepackage{mathrsfs}
\usepackage{graphicx}
\usepackage{bm} % for bold symbols in math mode \boldsymbol{}
\usepackage[nameinlink, capitalise]{cleveref}
\usepackage{tikz}
\usepackage{etoolbox}
\def\hB{\hspace*{\fill}$\qed$}% \newline\noindent}

\usepackage[nottoc]{tocbibind} % Bibliography im toc

\usepackage{defs_pp1}
\usepackage{slashed}
\usepackage[utf8]{inputenc}
\usepackage{microtype}
\usepackage[english]{babel}
\usepackage{mathtools}
\usepackage{bm}
\usepackage{esvect}

\title{A characterization of sheaves among six functor formalisms on $\LCH$.}
\author{
Ulrich Bunke\thanks{Fakult{\"a}t f{\"u}r Mathematik,
Universit{\"a}t Regensburg,
93040 Regensburg,
Germany\newline
ulrich.bunke@mathematik.uni-regensburg.de}  and 
Marco Volpe  \thanks{Fakult{\"a}t f{\"u}r Mathematik,
Universit{\"a}t Regensburg,
93040 Regensburg,
Germany\newline
marco.volpe@mathematik.uni-regensburg.de}
}

\numberwithin{equation}{section}
\newtheorem{theorem}{Theorem}[section] 
\newtheorem{prop}[theorem]{Proposition}
\newtheorem{lem}[theorem]{Lemma}

\newtheorem{ddd}[theorem]{Definition}
\newtheorem{kor}[theorem]{Corollary}

\theoremstyle{remark}
\theoremstyle{definition}

\newtheorem{ex}[theorem]{Example}
\newtheorem{rem}[theorem]{Remark}
\newcommand{\ICA}{\mathrm{cIdem}}

\newcommand{\desc}{\mathrm{desc}}
\newcommand{\CH}{\mathbf{CH}}

\newcommand{\CoSh}{\mathrm{CoShv}}

\newcommand{\Shv}{\mathrm{Shv}}

\newcommand{\All}{ \mathcal{A}\mathrm{ll}}

\newcommand{\LCH}{\mathbf{LCH}}

\newcommand{\cK}{\mathcal{K}}

\newcommand{\CAlg}{{\mathrm{CAlg}}}

\newcommand{\cD}{{\mathcal{D}}}

 \newcommand{\Cat}{{\mathbf{Cat}}}

\newcommand{\Open}{{\mathbf{Open}}}

\newcommand{\Span}{\mathrm{Span}}

\newcommand{\Spc}{\mathbf{Spc}}

\renewcommand{\Pr}{\mathbf{Pr}}

\newcommand{\op}{\mathrm{op}}

\renewcommand{\Spc}{\mathbf{An}}

\newcommand{\exa}{\mathrm{ex}}

\newcommand{\Frame}{\mathbf{Frame}}
 \newcommand{\Poset}{\mathbf{Poset}}

\newcommand{\st}{\mathrm{st}}
\renewcommand{\Open}{\mathrm{Open}}

\newcommand{\pt}{\mathrm{pt}}

\begin{document}  \maketitle 
  
  \begin{abstract}  
   Let $\cC$ be any stable presentably symmetric monoidal $\infty$-category. In this paper, we characterize $\Shv(-,\cC)$ on locally compact Hausdorff spaces as the unique six functor formalism satisfying a list of very natural properties. As a consequence, we deduce that every continuous six functor formalism $D$ in the sense of Zhu is equivalent to $\Shv(-, D(\pt))$. 
   %We show that the  six functor formalism $\Shv(-,E)$ on locally compact Hausdorff spaces is uniquely characterized amongst six functor formalisms $D$ with base ring $D(\pt)=E$ which have  canonical descent,  are section-determined and   whose  associated homology theory is continuous.
   \end{abstract} 

  \tableofcontents
  
     \section{Introduction}

The notion of six functor formalism was introduced by Grothendieck and his school in \cite{artinGrothendieckVerdierSGA4} as a framework encoding the formal properties of cohomology theories. Since then, six functor formalisms have appeared throughout topology, algebraic geometry and representation theory, providing a common categorical structure underlying many apparently different setups. Recent developments of \cite{arXiv:2510.26269, arXiv:2410.13038, Cnossen:2025aa} have shed new light on the foundations of the theory of six functor formalisms, in particular by providing a precise formal definition of what a six functor formalism is. One of the advantages of these new perspectives is that they provide effective tools to compare different six functor formalisms.

In topology, and more specifically on $\LCH$, the category of locally compact Hausdorff spaces, a prominent instance of a six functor formalism is given by $\Shv(-, \cC)$, i.e. sheaves valued in a sufficiently nice stable $\infty$-category $\cC$. We refer to \cite{Volpe2025SixOperations} for a development of the latter in the setting of $\infty$-categories, and to \cite{Verdier1965, KashiwaraSchapira1990} for more classical references. To the best of our knowledge, aside from $\Shv(-, \cC)$, there are only a few natural examples of six functor formalisms on $\LCH$. For this reason, one is naturally led to ask whether any six functor formalism on $\LCH$ satisfying a list of expected properties must already arise from sheaves. In order to formulate this question more precisely, let us review some of the most notable features of $\Shv(-,\cC)$.

The first and most apparent feature is that $\Shv(-,\cC)$ forms a sheaf valued in $\Cat$. Moreover, $\Shv(-,\cC)$ is localizing: any closed-open decomposition of a space $X$ induces a stable recollement of $\Shv(X,\cC)$. Following \cite{arXiv:2507.13537}, we refer to these two properties collectively as canonical descent. Since it is a subcategory of presheaves, equivalences in $\Shv(X,\cC)$ can be checked after evaluating at each open $U$ in $X$. We name this property as being section-determined. A less immediate but nevertheless essential feature is that sheaves are finitary. This means that $\Shv(-,\cC)$ preserves cofiltered limits of compact Hausdorff spaces when viewed as %when seen as 
a $\Pr_{\st}^{L,\op}$-valued functor via the pullback functors. As a consequence, sheaf cohomology is finitary as well. For a general six functor formalism $D$, we write $\Gamma^D$ for its associated cohomology theory $X\mapsto p_{X,*} p_X^* 1$, where $1$ in $D(\pt)$ is the monoidal unit and $p_X:X\to\pt$ is the unique map.

The main result of this paper shows that, under a mild additional assumption, these properties suffice to characterize sheaves amongst all six functor formalisms on $\LCH$.

\begin{theorem}[\cref{thkoperthertegrtger}]\label[theorem]{introthm:main}
	Let $D$ be a six functor formalism on the Nagata context $(\LCH,I,P)$. Assume that $D$ satisfies the following properties:
	\begin{enumerate}
		\item $D$ has %satisfies 
		canonical descent;
		\item $D$ is section-determined;
		\item $\Gamma^{D}$ is finitary;
		\item $D(\pt)$ is dualizable.
	\end{enumerate}
	Then $D$ is canonically equivalent to $\Shv(-,D(\pt))$.
\end{theorem}

The assumption that $D(\pt)$ is dualizable may be replaced by the requirement that $\Gamma^D$ is homotopy invariant, see \cref{rem:htpyinsteadofdual} and \cref{rem:mainthmhtpyinv}. 

It is well-known that for sheaves valued in higher categories, equivalences can not always be checked %of sheaves may be checked 
on stalks. However, at least when $\cC$ is dualizable and $X$ is % so-called
 hypercomplete, equivalences in $\Shv(X,\cC)$ can be checked stalkwise. We refer to this property as to being stalk-determined. As a second result we show that, under the finitary assumption, section-determinedness follows from stalk-determinedness.

\begin{prop}\label[prop]{introprop:stalk}
	Let $D$ be a six functor formalism on the Nagata context $(\LCH,I,P)$. Assume that $D$ has %satisfies 
	canonical descent, is finitary, and is stalk-determined. Then $D$ is section-determined.
\end{prop}

We conclude this introduction by mentioning some applications of our results. An immediate consequence of \cref{introthm:main} and \cref{introprop:stalk} is that any continuous six functor formalism on $\LCH$ in the sense of \cite{arXiv:2507.13537} is equivalent to $\Shv(-,D(\pt))$. In \cite{bunke2026theory}, the first %named 
author introduces a version of $E$-theory of $C^{*}$-algebras parametrized by arbitrary topological spaces, or more generally by locales. \cref{introthm:main} is used there to show that the $E$-theory  of $C^{*}$-algebras over a locally compact Hausdorff space is equivalent to sheaves valued in $E$.

 \subsection{Acknowledgments} Some of the ideas presented in this paper originated in the context of an ongoing collaboration of MV with Maxime Ramzi and Sebastian Wolf. MV especially thanks Maxime for discussions related to the comparison between a six-functor formalism D and sheaves. 
  
  \section{Definitions}%\label{koperthrtgerget}
  
The goal of this section is twofold. First, we recall the notion of a six functor formalism. Second, we introduce the properties of a six functor formalism $D$ that will later  %show  
imply that $D$ is equivalent to $\Shv(-,D(\pt))$. We also study the relations among these properties.

\subsection{Definition of a six functor formalism}

We begin by recalling the definition of a six functor formalism. We consider the Nagata set-up $(\LCH,I,P)$ on  the category $\LCH$ of locally compact Hausdorff spaces with the  subclasses of morphisms  
\begin{itemize}
%\item $E$: all morphisms
\item $I$ - the  open inclusions  
\item $P$  - the proper maps.
\end{itemize}
By  $\Pr^{L}_{\st}$ we denote the symmetric monoidal $\infty$-category of presentable stable $\infty$-categories and left-adjoint functors equipped with the Lurie tensor product. We write $ \CAlg(\Pr^{L}_{\st})
$ for the $\infty$-category of presentably symmetric monoidal stable $\infty$-categories and symmetric monoidal left-adjoint functors.
We consider a functor
\begin{equation}\label{bsdofjiovsfvsdfv}D:\LCH^{\op}\to \CAlg(\Pr^{L}_{\st})\ .
\end{equation} 
We will usually denote the contravariant functoriality of $D$ on morphisms $f$ in $\LCH$ by $f^{*}$, but sometimes we also use the notation $D(f)$, in particular when there are different functors floating arround.
Since $D$ takes values in presentable categories and colimit-preserving functors,  the functors $f^{*}$ have right-adjoints which will be denoted by $f_{*}$.

In this section we  formulate various conditions on $D$ and discuss relations between them.

\begin{ddd}\label[ddd]{okhpertgertgrtegetrhehhrt}
We say that $D$ is compatible with $I$ if:\begin{enumerate}
\item ( I-la: left adjoints) For every $i:U\to X$ in $I$ the functor $i^{*}$ admits a left-adjoint $i_{!}:D(U)\to D(X)$.
\item\label{okhpertgertgrtegetrhehhrt111}  (I-bc: base change)  For every cartesian square $$ \xymatrix{U\ar[r]^{i}\ar[d]^{g} &X \ar[d]^{f} \\ V\ar[r]^{j} &Y } $$ in $\LCH$ with $i,j$ in $I$
the square $$\xymatrix{ D(Y)\ar[r]^{j^{*}}\ar[d]^{f^{*}} & D(V)\ar[d]^{g^{*}} \\D(X) \ar[r]^{i^{*}} & D(U)} $$
is horizontally left-adjointable, i.e.  the  canonical Beck-Chevalley map $ i_{!}g^{*}  \to  f^{*}j_{!} $ is an equivalence.
 \item\label{okhpertgertgrtegetrhehhrt1111} (I-pf: projection formula) For every $i:U\to X$ in $I$ the canonical map $$i_{!}(-\otimes_{U} i^{*}(-))\to i_{!}(-)\otimes_{X} i^{*}(-)$$ is an equivalence.
 \end{enumerate}
\end{ddd}

\begin{ddd}\label[ddd]{zpozhrtzhztjrzthtzhtrh}
We say that $D$ is compatible with $P$ if:\begin{enumerate}
\item\label{zpozhrtzhztjrzthtzhtrh1} (P-ra: right adjoint) For every $p:X\to Y$ in $P$ the right-adjoint $p_{*}:D(X)\to D(Y)$ of the functor $p^{*} $  is cocontinuous.  
\item (P-bc: base change)  For every cartesian square $$ \xymatrix{W\ar[r]^{f}\ar[d]^{q} &X \ar[d]^{p} \\ Z\ar[r]^{g} &Y } $$ in $\LCH$ with $p,q$ in $P$
the square $$\xymatrix{ D(Y)\ar[r]^{g^{*}}\ar[d]^{p^{*}} & D(Z)\ar[d]^{q^{*}} \\D(X) \ar[r]^{f^{*}} & D(W)} $$
is vertically right-adjointable, i.e. the  canonical Beck-Chevalley map $ g^{*}p_{*}   \to q_{*} f^{* } $ is an equivalence.
 \item (P-pf: projection formula) For every $p:X\to Y$ in $P$ the canonical map $$p_{*}(-)\otimes_{Y}(-) \to p_{*}(-\otimes_{X} p^{*}(-))$$ is an equivalence.
 \end{enumerate}
\end{ddd}

If $D$ is compatible with $I$ and $P$, then we make the following definition:
 \begin{ddd}
We say that $D$ satisfies the mixed Beck-Chevalley (mBC) condition if 
for every   cartesian square $$ \xymatrix{W\ar[r]^{i}\ar[d]^{q} &X \ar[d]^{p} \\ Z\ar[r]^{j} &Y } $$ in $\LCH$ with $p,q$ in $P$
and $i,j$ in $I$ 
the square  $$\xymatrix{ D(Z)\ar[r]^{j_{!}}\ar[d]^{p^{*}} & D(Y)\ar[d]^{q^{*}} \\D(W) \ar[r]^{i_{!}} & D(X)} $$  given by I-bc
is vertically right-adjointable, i.e. the  canonical Beck-Chevalley map $    j_{!}p_{*}   \to  q_{*}  i_{!}   $ is an equivalence.
\end{ddd}

Let \begin{equation}\label{bwerfwerwregw}\cB:D\to E:\LCH^{\op}\to \CAlg(\Pr^{L}_{\st}) \end{equation}be a natural transformation.  
 In the following it is useful to write $D(i)$ instead of $i^{*}$.
\begin{ddd}\label[ddd]{kohperhretgertge} \mbox{}\begin{enumerate}\item \label{thokerpthertgertge} (I($\cB$))
We say that $\cB$ is compatible with $I$ if for every morphism $j:U\to X$ in $I$
the square
$$\xymatrix{D(X)\ar[r]^{\cB_{X}}\ar[d]^{D(j)} &E(X) \ar[d]^{E(j)} \\D(U) \ar[r]^{\cB_{U}} & E(U) } $$
is vertically left-adjointable, i.e. the  canonical Beck-Chevalley map $ E(j)_{!}B_{U}\to \cB_{X} D(j)_{!}$ is an equivalence.
\item  (P($\cB$)) We say that $\cB$ is compatible with $P$ if for every morphism $p:X\to Y$ in $P$ 
the square
$$\xymatrix{D(Y)\ar[r]^{\cB_{Y}}\ar[d]^{D(p)} &E(Y) \ar[d]^{E(p)} \\D(X) \ar[r]^{\cB_{X}} & E(X) } $$
is vertically right-adjointable, i.e. the  canonical Beck-Chevalley map $B_{X} E(p)_{*}\to D(p)_{*}\cB_{Y} $ is an equivalence.
\end{enumerate}
\end{ddd}

\begin{ddd} \mbox{}\label[ddd]{kojpjrtzjzhrtzh}\begin{enumerate}\item 
A six functor formalism  on the Nagata context $(\LCH,I,P)$ is a functor
$D$ as in \eqref{bsdofjiovsfvsdfv} that is compatible with $I$, $P$ and satisfies the mBC-condition.
\item A morphism between six functor formalisms on the Nagata context $(\LCH,I,P)$ is  a natural transformation
  as   in \eqref{bwerfwerwregw} 
that is compatible with $I$ and $P$.
\end{enumerate}
\end{ddd}

\begin{rem}
A reader familiar with some of the recent literature on six functor formalisms might think that our definition does not match with the one provided in \cite{arXiv:2410.13038, arXiv:2510.26269}. We now explain why the two approaches are equivalent. Thanks to the work of \cite{liu2012gluing, Cnossen:2025aa}, a six functor formalism $D$ on $(\LCH,I,P)$ as in \cref{kojpjrtzjzhrtzh} gives rise to a lax symmetric monoidal functor from the span category $\Span(\LCH,\All)$ to $\Pr^{L}_{\st}$, i.e. a six functor formalism in the sense of \cite{arXiv:2410.13038, arXiv:2510.26269}. For every morphism $f:X\to Y$ in $\LCH$, this extension gives in particular a functor $f_{!}:D(X)\to D(Y)$ with right adjoint $f^{!}$, satisfying $f_{!}\dashv f^{*}$ whenever $f\in I$, and 
$f_{!}\simeq f_{*}$ whenever $f\in P$. By virtue of \cite[Theorem 3.3]{dauser2024uniqueness}, such extension is uniquely determined from $D$. As a consequence, we see that the two approaches are equivalent. 
\hB
\end{rem}

%\begin{ddd}
%We let $$\sFF(\LCH,I,P)\subseteq \Fun(\LCH^{\op},\CAlg(\Pr^{L}_{\st}))$$
%denote the subcategory of 6 functor calculi and morphisms.
%\end{ddd}

%\fuli{add a remark that this is a same as 6 functor caluli defined on a span category etc, cite heyer-Mann, Scholze}

\subsection{Descent and finitariness}

In this subsection, we introduce some descent and finitariness conditions on a six functor formalism, that were also considered in \cite{arXiv:2507.13537}. 

Let $D:\LCH^{\op}\to \CAlg(\Pr^{L}_{\st})$ be a functor % as in \eqref{bsdofjiovsfvsdfv}
 that is compatible with $I$.
The following definition is taken from  \cite[Def. 4.11]{arXiv:2507.13537}.
\begin{ddd}\label[ddd]{kojghptrertgt}
We say that $D$ has  canonical  descent if it satisfies:
\begin{enumerate} 
\item (zero) \label[item]{hkopzetgerth}$D(\emptyset)\simeq 0$
\item (localization)  \label[item]{hkopzetgert1h} For every $X$ in $\LCH$, open subset $j:U\to X$ and closed complement $i:Z:=X\setminus U\to X$ we have a stable recollement
$$ \xymatrix{ D(U)\ar@/^0.8cm/[r]_{\perp}^{j_{!}}\ar@/_0.8cm/[r]^{\perp}_{j_{*}}&\ar[l]_{j^{*}}D(X)\ar@/^0.8cm/[r]_{\perp}^{i^{*}}\ar@/_0.8cm/[r]^{\perp}_{i^{!}}&\ar[l]_{i_{*}}D(Z)}\ .$$
 \item  (open exhaustions)
For every filtered family $(U_{i})_{i\in I}$ in $\Open(X)$ with $U:=\bigcup_{i\in I} U_{i}$ we have an equivalence
\begin{equation}\label{giojewroigwerfewrfwer}D(U)\stackrel{\simeq}{\to} \lim_{i\in I^{\op}} D(U_{i})
\end{equation} in $\Pr^{L}_{\st}$ induced by contravariant functoriality for the inclusions.
 \hB
\end{enumerate}
If $D$ satisfies only the first two conditions, we say that $D$ is localizing.
\end{ddd}

\begin{rem} Compatibility of $D$ with $I$  first of all ensures by I-la that $j_{!}$ exists.
Condition \ref{hkopzetgerth} together with $U\cap Z=\emptyset$  and  I-bc ensure that $i^{*}\circ j_{!}\simeq 0$. %The  
Condition \ref{hkopzetgert1h} then requires that
 the sequence $$D(U)\stackrel{j_{!}}{\to} D(X)\stackrel{i^{*}}{\to} D(Z)$$ is a fibre and cofibre sequence in $\Pr^{L}_{\st}$. 
 As a consequence we get a cofibre sequence of endofunctors
 \begin{equation}\label{bwiojoievewrvfds}j_{!}j^{*}\to \id\to i_{*}i^{*}
\end{equation} 
of $D(X)$.
Furthermore, the pair of functors $(j^{*},i^{*})$ is jointly conservative. 
 % and that
% $j_{!}$ and $i_{*}$ are fully faithful. Then also $D(Z)\stackrel{i_{*}}{\to} D(X)\stackrel{j^{*}}{\to} D(U)$
% is a fibre sequence in $\Pr^{L}_{\st}$.
  \hB
\end{rem}

As a  left adjoint of the lax symmetric monoidal functor $j^{*}$, the functor $j_{!}$ is canonically oplax symmetric monoidal.
In what follows,   we will show that localization for $D$ implies that $j_{!}$ is actually strong 
oplax symmetric monoidal.
%As a  left adjoint of the lax symmetric monoidal functor $j^{*}$, the functor $j_{!}$ is canonically oplax symmetric monoidal.
%In what follows,   we will show that localization for $D$ implies that $j_{!}$ is actually strong 
%oplax symmetric monoidal.
%The assumption of localization for $D$ makes\fuli{I do not understand what this has to do with localization?} $j_!$ for any $j$ in $I$ into a strong oplax symmetric monoidal functor in the following sense.
%\textcolor{blue}{I want to leave at least this to justify why we are giving the following definition, otherwise it looks like it appeared out of nowhere.}

\begin{ddd} 
	Let $\cC$, $\cD$ be two symmetric monoidal $\infty$-categories. A oplax monoidal functor $F:\cC\to\cD$ is called strong if, for every $C_1,C_2\in\cC$, the map $$F(C_1\otimes C_2)\to F(C_1)\otimes F(C_2)$$
	given by the oplax structure is invertible.
\end{ddd}

We write $\CAlg(\Pr^{L})_{\mathrm{soplax}}$ (or, $\CAlg(\Pr^{L}_{\st})_{\mathrm{soplax}}$, respectively) for the $\infty$-category of presentably symmetric monoidal (stable) $\infty$-categories and strong oplax symmetric monoidal functors between them.
 
 \begin{lem}\label[lem]{ggjiwoergergewfrefwref}
Let $D:\LCH^{\op}\to \CAlg(\Pr^{L}_{\st})$ be a functor %as in \eqref{bsdofjiovsfvsdfv} 
that
%which 
is compatible with $I$. Assume that $D$ is localizing. Let $X$ be in $\LCH$, and $j:U\to X$ be the inclusion of an open subset. Then the functor $j_!:D(U)\to D(X)$ is strong oplax symmetric monoidal.
\end{lem}

\begin{proof}
	Since the functor  $j^{*}$ is symmetric monoidal, its left-adjoint $j_{!}$ is oplax symmetric monoidal. We must show that the  op-lax monoidal structure map $$j_{!}(-\otimes -)\to j_{!}(-)\otimes j_{!}(-)$$  is an equivalence. Since $(j^{*},i^{*})$ is jointly conservative (with $i:X\setminus U\to X$ the inclusion of the complement), it suffices to observe  that it becomes an equivalence after application of $j^{*}$ and $i^{*}$.
	If we apply $j^{*}$ and use that this functor is symmetric monoidal and $j^{*}j_{!}\simeq \id$, then we get the identity map.
	If we apply $i^{*}$, then using that this functor is symmetric monoidal again and $i^{*}j_{!}\simeq 0$ we get
	an equivalence between zero objects.
\end{proof}

\begin{rem}
Note that the limit in \eqref{giojewroigwerfewrfwer} can also be interpreted in $\Cat^{\exa}_{\infty}$ since
the functor $\Pr^{L}_{\st}\to \Cat^{\exa}_{\infty}$ preserves and reflects limits.
The exhaustion condition  implies   equivalences \begin{equation}\label{gwerfwrgfdfg}\colim_{i\in I} j_{i,!} j_{i}^{*}\to j_{  !} j_{  }^{*}  
\end{equation}
and
\begin{equation}\label{gwerfwfeerferrgfdfg} j_{ *}j_{ }^{*}\to \lim_{i\in I^{\op}} j_{i,* } j_{i}^{*}\end{equation}
of endofunctors of $D(X)$,
where $j:U\to X$ and $j_{i}:U_{i}\to X$ are the inclusions.
 \hB
\end{rem}

 The following notion is taken  from  \cite[Def. 3.1]{arXiv:2507.13537}.
\begin{ddd}\label[ddd]{kohperthertgertg}\mbox{}\begin{enumerate}
\item 
A functor $D:\LCH^{\op}\to \CAlg(\Pr^{L}_{\st})$ %as in   \eqref{bsdofjiovsfvsdfv} 
is called a coefficient system 
 if it is compatible with $I$ (see \cref{okhpertgertgrtegetrhehhrt}) and satisfies canonical descent (in the sense of \cref{kojghptrertgt}). 
 \item A morphism between coefficient systems
 is a natural transformation of functors as in   \eqref{bwerfwerwregw} 
 which is compatible with $I$ (in the sense of \cref{kohperhretgertge}.\ref{thokerpthertgertge}).
  \end{enumerate}
\end{ddd}

%\begin{rem}\uli{remove?}
%In the literature of 6 functor calculi  one also considers another descent condition called sheafyness.
%\begin{ddd}
%A 6 functor calculus $D$ is called sheafy if $D$ is a sheaf with values in $\Cat^{\exa}_{\infty}$.
%\end{ddd}
%In contrast to canonical descent sheafyness requires instead of the localization condition  \cref{kojghptrertgt}.\ref{kohperhgertgert}
%the condition of descent for open coverings by pairs of open subsets.
%If $D$ satisfies canonical descent then it is sheafy.   \hB
% \end{rem}
 
 We consider a functor $D:\LCH^{\op}\to \CAlg(\Pr^{L}_{\st})$. %as in   \eqref{bsdofjiovsfvsdfv}.
 Let $\CH\subseteq\LCH$ be the full subcategory   of compact Hausdorff spaces.
 \begin{ddd}\label[ddd]{hkerpohertgertgtr}
 We say that $D$ is finitary if for every cofiltered system $(X_{i})_{i\in I}$ in $\CH$
 with $X:=\lim_{i\in I}X_{i}$   we have
 an equivalence
 $$  \colim_{i\in I^{\op}} D(X_{i})\stackrel{\simeq}{\to} D(X)\ ,$$
 where the colimit is interpreted in $\Pr^{L}_{\st}$ for the contravariant  functoriality.  
 \end{ddd}
 
 \begin{rem}\label[rem]{hkoperthevsfdvftrgertge}
 We can equivalently require that 
 $$ D(X)\stackrel{\simeq}{\to}  \lim_{i\in I} D(X_{i})  $$
for the covariant  functoriality $f_{i,*}:D(X)\to D(X_{i})$,
where the limit is interpreted  in  $\Cat^{\exa}_{\infty}$. 
%Continuity 
Finitariness implies an equivalence
\begin{equation}   \colim_{i\in I^{\op}} f_{i}^{*}f_{i,*}\stackrel{\simeq}{\to} \id_{D(X)}\ ,\end{equation}
of endofunctors of $D(X)$ and that the family
$(f_{i,*}:D(X)\to D(X_{i}))_{i\in I}$ is jointly conservative.
\hB
 \end{rem}
 
 \subsection{Sections and cohomology}

In this subsection, we define the functors of evaluation at an open subset and the cohomology theory associated with any six functor formalism. We define $D$ to be section-determined at a space $X$ if the family of evaluation funtors indexed by all opens in $X$ is jointly conservative. We also observe how some of the properties introduced earlier on $D$ imply analogous properties on the cohomology theory associated with $D$. 

Let $D:\LCH^{\op}\to \CAlg(\Pr^{L}_{\st})$ be a functor. %as in     \eqref{bsdofjiovsfvsdfv}.
\begin{ddd}\label[ddd]{koptrhrtegtergrtgertg} For $X$ in $\LCH$ and  $U$  in $\Open(X)$,   
we define the evaluation 
$$\ev_{U}:=p_{U,*} j_{U\to X}^{*}:D(X)\to D(\pt)\ .$$
 \end{ddd} Here 
  $p_{U}:U\to \pt$  is the projection    and $j_{U\to X}:U\to X$ is the inclusion.
  For $A$ in $D(X)$, we call $\ev_{U}(A)$ the  sections of $A$ on $U$.
  
 \begin{ddd}\label[ddd]{lkpohejztjrtz} \mbox{}
	\begin{enumerate} \item
		We say that $D$ is section-determined on $X$ if the collection $(\ev_{U})_{U\in \Open(X)}$ of evaluation   maps  is jointly conservative. \item  We say that $D$ is section-determined if it is section-determined on every $X$ in $\LCH$.
	\end{enumerate}
\end{ddd}

\begin{rem}
	In the case of $D=\Shv(-,E)$ the evaluation $\ev_{U}$ sends a sheaf $F$ to its value on $U$, i.e. we have  $\ev_{U}(F)\simeq F(U)$.  In this case, $D$ is tautologically  section-determined.
	\hB
\end{rem}
  
  Consider a functor  $D:\LCH^{\op}\to \CAlg(\Pr^{L}_{\st})$.
  %Assume that $D$ is as in   \eqref{bsdofjiovsfvsdfv}.  
   \begin{ddd}\label[ddd]{okphertgertgerg}
  	We   define the  cohomology functor associated with  $D$ by
  	$$\Gamma^{D}:\LCH^{\op}\to D(\pt)\ , \quad X\mapsto p_{X,*} p_{X}^{*}1\ .$$\end{ddd}
  Here 
  $1$ is  the monoidal unit %tensor unit 
   in $D(\pt)$. 
  %\fuli{why functor, we have used this before, e.g. in the definition of the stalk. How do we deal with this problem in general?}
  \begin{rem}
  	If $f:X\to Y$  is a morphism  in $\LCH$, then the contravariant functoriality for $\Gamma^{D}$ is given by
  	$$\Gamma^{D}(f):\Gamma^{D}(Y)\simeq p_{Y,*} p_{Y}^{*}1\to p_{Y,*} f_{*}f^{*}p_{Y}^{*}1 \simeq 
  	p_{X,*}p_{X}^{*}1\simeq \Gamma^{D}(X)$$ using the unit of the adjunction $f^{*}\dashv f_{*}$
  	and $p_{Y}\circ f=p_{X}$. See \cite[Proposition 3.3]{Volpe2025SixOperations} for details on how to make this construction fully coherent.
  	Furthermore note that $$\Gamma^{D}(X)\simeq \ev_{X}(1_{X}).$$
  	\hB
  \end{rem}
  
  % \fuli{removed strong excision, since the cphomology does not satisfy this}
  
  Let $\cC$ be a stable $\infty$-category and $H:\CH^{\op}\to \cC$ be some functor.
  \begin{ddd}\label[ddd]{kohperthrtgeg}
  	We say that $H$ satisfies   strong  excision if for every  map $f:(X,A)\to (Y,B)$ of pairs in $\CH$ that induces a homeomorphism
  	$X\setminus A\stackrel{\cong}{\to} Y\setminus B$ the  square
  	$$\xymatrix{H(Y)\ar[r]\ar[d]&H(B)\ar[d]\\H(X)\ar[r]&H(A)}$$
  	is cartesian.
  \end{ddd}
  %  %   \begin{ddd}\label{kohperthrtgeg}
  	%  We say that $H$ satisfies   closed excision if for every $X$ and closed subsets $Y,Z$ such that $Y\cup Z=X$     the  square
  	%  $$\xymatrix{H(X )\ar[r]\ar[d]&H(Y)\ar[d]\\H(Z)\ar[r]&H(Y\cap Z)}$$
  	%  is cartesian.
  	%  \end{ddd}
  
  %  \fuli{removed the lemma trying to verrify strong excision fo $\Gamma^{D}$ since not provable, since the cphomology does not satisfy this}
  \begin{lem}\label[lem]{kopthrthertgertg}
  	If $D$ is  a six functor formalism that satisfies canonical descent, then the restriction of  $\Gamma^{D}$ to compact Hausdorff spaces satisfies strong excision. \end{lem} 
  \begin{proof} 
  	% Let $X$ be a compact Hausdorff space, $U$ be  an open subset of $X$ and $Z:=X\setminus U$ be its closed complement. 
  	% Strong excision requires that the fibre $\Fib(H_{D}(X)\to H_{D}(Z))$ 
  	% of the restriction map along the inclusion $Z\to X$ only depends on the complement $U$, see below for the precise statement.  
  	We use the notation from \cref{kohperthrtgeg} and set $U:=X\setminus A$ and $V:=Y\setminus B$.
  	By  localization (specializing \eqref{bwiojoievewrvfds})  we get
  	the map of fibre sequences
  	%$$p_{X,*}j_{U\to X,!} j_{U\to X}^{*}p_{X}^{*}1 \to p_{X,*}p_{X}^{*} 1 \to  p_{X,*}i_{A\to X,*}i_{A\to X}^{*}p_{X}^{*}1\ .$$
  	$$\xymatrix{ p_{Y,*}j_{V\to Y,!} j_{V\to Y}^{*}p_{Y}^{*}1\ar[r]\ar@{..>}[d] &p_{Y,*}p_{Y}^{*} 1\ar[r]\ar[d] & p_{Y,*}i_{B\to Y,*}i_{B\to Y}^{*}p_{Y}^{*}1 \ar[d] \\ p_{X,*}j_{U\to X,!} j_{U\to X}^{*}p_{X}^{*}1 \ar[r]&p_{X,*}p_{X}^{*} 1\ar[r] & p_{X,*}i_{A\to X,*}i_{A\to X}^{*}p_{X}^{*}1    }\ , $$
  	where for the vertical maps we use relations like 
  	$p_{Y}f=p_{X}$  and units like
  	$\id\to f_{*}f^{*}$.
  	The dotted arrow is the equivalence 
  	$$p_{X,*}j_{U\to X,!}   j_{U\to X}^{*}p_{X}^{*}1\simeq p_{Y,*}f_{*}j_{U\to X,!}  j_{U\to X}^{*} f^{*}p_{Y}^{*} 1\stackrel{!}{\simeq} p_{Y,*} j_{V\to Y,! }f_{|U,*} f_{|U}^{*} j_{V\to Y}^{*}p_{Y}^{*}1\stackrel{!!}{\simeq}  p_{Y,*} j_{V\to Y,! } j_{V\to Y}^{*}p_{Y}^{*}1$$
  	where we used the mBC condition for the cartesian square  $$\xymatrix{U\ar[r]^{j_{U\to X}}\ar[d]^{f_{|U}} &X \ar[d]^{f} \\V \ar[r]^{j_{V\to Y}} &Y } $$
  	at the equivalence marked by $!$ and that $f_{|U}$  is a homeomorphism 
  	for the equivalence marked by $!!$.
  	% 
  	% Using that  mBC  the fibre of the second map can be rewritten in the form 
  	%$H_{D,c}(U):=p_{U,!} p_{U}^{*}1$, where $H_{D,c}(U)$ is called  the compactly supported cohomology of $U$. We thus get the fibre sequence
  	%$$H_{D,c}(U)\to H_{D}(X)\to H_{D}(Z)\ .$$
  	%If $g:X'\to X$ is a map of  compact Hausdorff spaces, then we set $Z':=g^{-1}(Z)$ and $U':=g^{-1}(U)$. We  assume that $g_{|U'}:U'\to U$ is a homeomorphism.
  	%Then we get a morphism of fibre sequences
  	%$$\xymatrix{H_{D,c}(U)\ar[r]\ar[d]^{\simeq}&H_{D}(X)\ar[r]\ar[d]^{g^{*}}&H_{D}(Z)\ar[d]^{g_{|Z}^{*}}\\H_{D,c}(U')\ar[r]&H_{D}(X')\ar[r]&H_{D}(Z')}\ .$$
  	%Strong excision  expresses the fact that in this geometric situation the left vertical map is always an equivalence.
  	%   \end{proof}
  %% 
  %  
  %\begin{lem}\label{kopthrthertgertg}
  %  If $D$ is compatible  with $I$  and satisfies canonical descent, then the restriction of  $\Gamma^{D}$ to compact Hausdorff spaces satisfies  closed excision. \end{lem} 
  %  \begin{proof} 
  %  By  localization (specializing \eqref{bwiojoievewrvfds}), and using $U:=X\setminus Y$ and $V:=Z\setminus Y\cap Z$  we get a map of fibre sequences (using $p_{Z}=p_{X}\circ i_{Z\to X}$ and the unit of $i_{Z\to X}^{*}\dashv  i_{Z\to X,*}$)
  %  $$\xymatrix{p_{X,*}j_{U\to X,!} j_{U\to X}^{*}p_{X}^{*}1\ar[r]\ar@{..>}[d]&p_{X,*}p_{X}^{*}1 \ar[r]\ar[d]&p_{X,*} i_{Y\to X,*}i_{Y\to X}^{*}p_{X}^{*}1\ar[d]  \\ p_{X,*} i_{Z\to X,*}j_{V\to Z,!}j_{V\to Z}^{*}i_{Z\to X}^{*}p_{X}^{*}1\ar[r]&p_{X,*} i_{Z\to X,*}i_{Z\to X}^{*}p_{X}^{*}1\ar[r]& p_{X,*}i_{Y\cap Z\to X,*} i_{Y\cap Z\to X}^{*}p_{X}^{*}1 }$$
  %  
  %   
  \end{proof}
  
  \begin{rem}\label[rem]{ghuwegerfwefwerf}
  The condition of closed excision for $H$ is the special case of strong excision applied to maps
  $(Z,Y\cap Z)\to (X,Y)$ for compact subspaces $Y$ and $Z$ of $X$. So strong excision implies closed excision.\hB
  \end{rem}

  Let  $H:\LCH^{\op} \to \cC $ be  some functor.    
  \begin{ddd}\label[ddd]{iugwoierwerfwefwef}\mbox{}
  \begin{enumerate}
  	\item We say that $H$ is finitary     if for every cofiltered system $(X_{i})_{i\in I}$   in $\CH$ 
  	with $X:=\lim_{i\in I}X_{i}$ we have
  	an equivalence
  	$$  \colim_{i\in I^{\op}} H(X_{i})\stackrel{ \simeq}{\to} H(X)  \ .$$ % induced by the maps $H_{D}(f_{i})$.
  	\item We say that  $H $ is homotopy invariant if for every $X$ in $\LCH$ 
  	the projection $p:[0,1]\times X\to X$ induces an equivalence $H(p):H(X)\stackrel{\simeq}{\to} H([0,1]\times X)$.
  \end{enumerate}
  \end{ddd}
  
  %Let $D$ be  as in   \eqref{bsdofjiovsfvsdfv}.   
  %
  %\begin{ddd}  \begin{enumerate}
  %\item We say that $D$ satisfies  profinite cohomology descent if
  %$H_{D}$  satisfies  profinite descent.
  %  \item We say that $D$ is cohomologically homotopy invariant if   $H_{D}$ is homotopy invariant. \end{enumerate}
%
%\end{ddd}
%

\begin{ex}\label[ex]{kopherhtgertg}
Assume that $E$ is in $\CAlg(\Pr_{\st}^{L})$.
Then we have the 
six functor  formalism $$\Shv(-,E):\LCH^{\op}\to \CAlg(\Pr^{L}_{\st})$$ on the Nagata context $(\LCH,I,P)$.
It satisfies canonical descent and is finitary, hence fd-determined (see \cref{jzztjrthtzhrzth} below)  by \cref{okhpetrhrferferferfetgerge}.\ref{ijgohwphrtherthrthe1}. It is furthermore
section-determined and  stalk-determined. 
Its associated cohomology theory $\Gamma^{\Shv(-,E)}$ is   homotopy invariant and  finitary. The latter assertion  is the second statement   in \cite[Cor. 3.6.10]{NKP}. 

We will  write $\hat f^{*}$, $\hat f_{!}$, $\hat \otimes$ etc for the six functor operations  on sheaves. 
\hB
\end{ex}

\begin{prop}\label[prop]{hkoptrhertgertger}
If a  six functor formalism $D$ on the Nagata context $(\LCH,I,P)$  is finitary, then  % \uli{and $D(\pt)$ is dualizable}\fuli{this is what they assume, can we drop this}, 
then its associated cohomology functor $\Gamma^{D}$
is also finitary.
\end{prop}
\begin{proof}%[Proof of \cref{hkoptrhertgertger}] \phantomsection{}\label{goijwerogwreferfwefewrf}
For $D(-)=\Shv(-,D(\pt))$ this has been observed in \cref{kopherhtgertg}.       For general $D$, arguing as in \cite[Proposition 3.6.10]{NKP}, one can deduce that $\Gamma^D$ is finitary from \cite[Lemma. 2.5.10]{NKP}. %\uli{Achtung: They assume dualizablity of all values. This is more than we want to assume} 
%\fuli{need more details}
\end{proof}
  
\subsection{Stalks}

This subsection focuses on stalks. For each $D$ and $X$ in $\LCH$ with a point $x$ in $X$, we define the stalk functor as a pullback along the inclusion $i_x :\pt\rightarrow X$. We show that the usual colimit formula for stalks holds for any $D$ which is compatible with $I$ and $P$ and satisfies canonical descent. We introduce the condition of being stalk-determined for a six functor formalism $D$. This means that, for each $X$ in $\LCH$ which is hypercomplete, the family of stalk functors on $D(X)$ indexed by all points in $X$ is jointly conservative. We analyze the relation between being stalk-determined and section-determined. This will be important later to show that any continuous six functor formalism in the sense of \cite{arXiv:2507.13537} is equivalent to sheaves.

Consider a functor  $D:\LCH^{\op}\to \CAlg(\Pr^{L}_{\st})$.
  %Let  $D$ be a functor as in     \eqref{bsdofjiovsfvsdfv}. 
For $x$ in $X$ we let $i_{x}:\pt\to X$ denote the inclusion.
 \begin{ddd}\label[ddd]{okhprthegertgretg}   For $x$ in $X$ we 
 define the stalk at $x$ as the map
 $$i_{x}^{*} :D(X)\to D(\pt)\ .$$
   \end{ddd}

The following formula shows that, %\cref{okhprthegertgretg}
 when $D$ is sufficiently well-behaved, the usual formula computing stalks from sheaf theory is also valid for $D$.
\begin{lem} \label[lem]{triohjoertpghertgkl}If $D$ is  compatible with $I$ and $P$  and satisfies canonical descent, then
we have an equivalence \begin{equation}\label{gerfwerfrweffw}i_{x}^{*}\simeq \colim_{x\in V} \ev_{V} \ ,
\end{equation}  where the colimit runs over all open neighbourhoods of $x$.
\end{lem}
\begin{proof}
By localization and  using an  instance of \eqref{bwiojoievewrvfds} we have the fibre sequence
$$j_{X\setminus\{x\}\to X,!}j_{X\setminus \{x\}\to X}^{*}\to \id\to i_{x,*}i_{x}^{*}\ .$$
Since $X$ is locally compact Hausdorff we have the equality $X\setminus \{x\}=\bigcup_{x\in K} X\setminus K$, where $K$ runs over the compact subsets of $X$ which contain
an open neighbourhood of $x$.
By exhaustion, using an instance of  \eqref{gwerfwrgfdfg}, we get the equivalence
$$\colim_{x\in K}j_{X\setminus K\to X,!}j_{X\setminus K\to X}^{*}\simeq j_{X\setminus\{x\}\to X,!}j_{X\setminus \{x\}\to X}^{*}\ .$$
Using localization again we get 
$$i_{x,*}i_{x}^{*}\simeq \colim_{x\in K}i_{K\to X,*}i_{ K\to X}^{*}\ .$$
By a cofinality argument we  can   extend the index set  of the colimit first to all compact subspaces of $X$ which contain an open neighbourhood
of $x$ and then restrict to the open neighbourhoods. We get the equivalence
$$i_{x,*}i_{x}^{*}\simeq \colim_{x\in V}j_{V\to X,*}j_{ V\to X}^{*}\ .$$
We now apply $p_{X,*}$ and get 
$$i_{x}^{*}\simeq p_{X,*}\colim_{x\in V}j_{V\to X,*}j_{ V\to X}^{*}\ .$$
If $X$ is compact, then by P-ra (\cref{zpozhrtzhztjrzthtzhtrh}.\ref{zpozhrtzhztjrzthtzhtrh1}) the functor   $p_{X,*}$ 
preserves colimits and, using $p_{X,*}j_{V\to X,*}\simeq p_{V,*}$ we get the desired formula. 

 If $X$ is not compact, then we let $j_{X\to X^{+}}:X\to X^{+}$ be the inclusion into the one-point compactification.
We write $i_{x}^{+,*}$ for the stalk for $x$ considered as a point in $X^{+}$. Then by I-bc (\cref{okhpertgertgrtegetrhehhrt}.\ref{okhpertgertgrtegetrhehhrt111})
 we have $i_{x}^{+,*}j_{X\to X^{+},!}\simeq i_{x}^{*}$
 and $j_{V\to X^{+}}^{*}j_{X\to X^{+},!}\simeq j_{V\to X}^{*}$.
 So precomposing the already proven instance of the equivalence \eqref{gerfwerfrweffw} for $X^{+}$ with $j_{X\to X^{+},!}$ we get  the equivalence  \eqref{gerfwerfrweffw}.
   \end{proof}

\begin{rem}
Note that in the proof of \cref{triohjoertpghertgkl} we have not used the full power of the compatibility with P condition. We only used that $p_{X,*}$ is cocontinuous for  the projection maps
$p_{X}:X\to \pt$  for compact Hausdorff spaces. The same remark also applies to \cref{kohpertrtgertgrtgterg}
and \cref{okhpetrhrferferferfetgerge} below where the assumption of compatibility with P is added since their proofs use the formula  for the stalk from \cref{triohjoertpghertgkl}.
\hB
\end{rem}

A topological space $X$ is called hypercomplete if the topos $\Shv(X,\Spc)$ is hypercomplete in the sense of \cite[Sec. 6.5.2]{htt}. Hypercompleteness is equivalent to the condition
 that the family of stalks $(i_{x}^{*}:\Shv(X,\Spc)\to \Spc)_{x\in X}$   is  jointly conservative; see \cite[Lemma A.3.9]{HA}.

  %For $x$ in $X$ we let $i_{x}:\pt\to X$ denote the inclusion.
  %We consider a functor $D$ as in   \eqref{bsdofjiovsfvsdfv}.
 We consider a functor  $D:\LCH^{\op}\to \CAlg(\Pr^{L}_{\st})$.
\begin{ddd}\label[ddd]{kopherthrgertgtg} \mbox{}
 \begin{enumerate} \item
 We say that $D$ is stalk-determined  on $X$ if  the family of  stalks 
 $$(i_{x}^{*} :D(X)\to  D(\pt))_{x\in X}$$ 
 is  jointly conservative.
 \item We say that $D$ is stalk-determined if it is stalk-determined on   every hypercomplete locally compact Hausdorff space $X$.
  \end{enumerate}
 \end{ddd}

 \begin{lem}\label[lem]{kohpertrtgertgrtgterg}
 If $D$ is compatible with $I$ and $P$, satisfies canonical descent, and is stalk-determined on $X$, then it is section-determined on $X$.
 \end{lem}
 \begin{proof}
 Let $h:A\to B$ be a morphism in $D(X)$ and assume that $\ev_{U}(h)$ is an equivalence for every
 $U$ in $\Open(X)$.  In view of  \eqref{gerfwerfrweffw} we conclude that $i_{x}^{*}(h)$ is an equivalence for every $x$ in $X$. Since $D$ is stalk-determined on $X$, we conclude that $h$ is an equivalence.
 \end{proof}

  %We consider a functor $D$ as in   \eqref{bsdofjiovsfvsdfv}.
We consider a functor  $D:\LCH^{\op}\to \CAlg(\Pr^{L}_{\st})$.
\begin{ddd}\label[ddd]{jzztjrthtzhrzth}
 We say that $D$ is fd-determined if for every compact Hausdorff space $X$ the family  
 $(f:X\to Y)_{f}$ induces a jointly conservative family $(f_{*}:D(X)\to D(Y))_{f}$, where $f$ runs over all maps from $X$ to closed subspaces of $[0,1]^{N}$ for some $N$ in $\nat$.
\end{ddd}

\begin{rem}
Being fd-determined essentially means that for every $X$, the collection of pushforwards along maps from $X$  to finite-dimensional spaces is jointly conservative. \hB
\end{rem}
%  The following is a more precise version of \cref{wegwiojerogefwerfwref} from the introduction.
 %Let $D$ be a functor as in     \eqref{bsdofjiovsfvsdfv}.
 \begin{prop}\label[prop]{okhpetrhrferferferfetgerge}
  \mbox{}
  \begin{enumerate}
  \item \label[item]{ijgohwphrtherthrthe} If $D$  is compatible with $I$ and $P$, satisfies canonical descent, and is   fd-determined  and  stalk-determined, then it is section-determined.
  \item \label[item]{ijgohwphrtherthrthe1} If $D$ is finitary, %continuous
  then it is fd-determined.
  \end{enumerate}
\end{prop}
\begin{proof}%[Proof of \cref{okhpetrhrferferferfetgerge}]
%\phantomsection{}\label{okhpetrhretgerge} %For  \cref{okhpetrhrferferferfetgerge}.\ref{hkopertertg}
%let $X$ be a hypercomplete compact Hausdorff space.
%
%
%we use the localization
%
%
We start with Assertion \ref{ijgohwphrtherthrthe}.
%Assume that $D$ is compatible with $I$, satisfies canonical descent and  is continuous  and hyperlocal. 
Let $X$ be in $\LCH$. We must show that the collection of evaluations $(\ev_{U})_{U\in \Open(X)}$ is jointly conservative.
 By  localization   the functor $j_{X\to X^{+},*}$ is fully faithful. 
Since $\ev_{U}^{X}\simeq \ev_{U}^{X^{+}}\circ j_{U\to X^{+},*}
$ (we add the superscript in order to indicate   for which space the evaluation is considered) it suffices to show
that $(\ev^{X^{+}}_{U})_{U\in \Open(X^{+})}$ is jointly conservative.
We can therefore assume that $X$ is compact.

\begin{lem}\label[lem]{kohpertrtgertgetrg}
If $D$ is fd-determined and %section
stalk-determined on every closed subspace of $[0,1]^{N}$ for some $N$ in $\nat$,
then $D$ is section-determined.
\end{lem}
\begin{proof} By assumption,
the family of maps $f_{*}:D(X)\to D(Y)$, where $Y$ is a closed subspace of  $[0,1]^{N}$ for some $N$ in $\nat$,
is jointly conservative. For every such map and
open subset $V$ in $Y$, we have the pull-back
$$\xymatrix{U\ar[rr]^-{j_{U\to X}}\ar[d]_{f_{|U} }&&X\ar[d]^{f} \\ V\ar[rr]^-{j_{V\to Y}} &&Y} $$
with $U$ in $\Open(X)$.
By  I-bc   we have an equivalence
 $$ j_{U\to X,!}  f_{|U}^{*}\simeq f^{*}  j_{V\to Y,!} \ .$$
 Taking right-adjoints we get
 $$ j_{V\to X}^{*}   f_{*}\simeq f_{|U,*} j_{U\to X}^{*}\ .$$
Postcomposing with $p_{V,*}$ we conclude that
$$\ev_{V}\circ f_{*}\simeq \ev_{U} :D(X)\to D(\pt)\ .$$
We now use that $D$ is section-determined on $Y$ by assumption.
We see that the family $(\ev_{V}\circ f_{*})_{f:X\to Y,V\in \Open(Y)}$ is jointly conservative. This implies that $(\ev_{U})_{U\in \Open(X)}$ is jointly conservative.
\end{proof}

We can now finish the proof of Assertion \ref{ijgohwphrtherthrthe}.
Since closed subspaces $Y$  of $[0,1]^{N}$ are hypercomplete, according to our assumption
 $D$ is stalk-determined on those $Y$.  By \cref{kohpertrtgertgrtgterg} it is section-determined on those $Y$.
 The assertion now follows from \cref{kohpertrtgertgetrg}.

We now show Assertion \ref{ijgohwphrtherthrthe1}.  Let $X$ be a compact Hausdorff space.
We have the closed embedding $X\to [0,1]^{S}$, $x\mapsto (\phi(x))_{\phi\in S}$, where $S$ is the set of continuous functions
$\phi:X\to [0,1]$.  For every finite subset $F\subseteq S$ let $Y_{F}$ be the image of the composition $X\to [0,1]^{S}\to  [0,1]^{F}$, where the second map is the projection.  Then $ X\cong \lim_{F\subseteq S,|F|<\infty} Y_{F}$.   We let $f_{F}:X\to Y_{F}$ denote the structure maps. Finitariness %Continuity 
of $D$ implies,  as seen in \cref{hkoperthevsfdvftrgertge}, that
the family $(f_{F,*})_{F\subseteq S,|F|<\infty }$ is jointly conservative.
This implies that  $D$ is fd-determined.  
  \end{proof}

 %The proof of \cref{okhpetrhrferferferfetgerge} will be given in
%\cref{okhpetrhretgerge}.

\section{The main result}
          
The goal of this section is to give a proof of the main theorem of this paper, characterizing $\Shv(-, E)$ in terms of the properties introduced in the previous section. We deduce that any continuous six functor formalism $D$ in the sense of \cite{arXiv:2507.13537} must be equivalent to $\Shv(-, D(\pt))$.
 
\subsection{The transformation $\cB:\Shv(-)\rightarrow D$} 

Let $D$ be any coefficient system (see \cref{kohperthertgertg}). The goal of this subsection is to construct a map of coefficient systems $\Shv(-,D(\pt))\rightarrow D(-)$. To this end, %For that end 
we use the %sheaves 
sheaf-spectrum adjunction given in \cite{arXiv:2302.04069}.

We start by recalling some preliminary definitions. 
Recall that a poset is said to be a frame if it has all joins and finite meets, %and 
and finite meets distribute over joins.
A morphism of frames is a morphism of posets %which
that preserves finite meets and all joins.  We write $\Frame$ for the subcategory of $\Poset$ whose objects are frames and whose morphisms are frame morphisms.

The poset $\Open(X)$ of a topological space $X$ is a frame. 
If $f:X\to Y$ is a continuous map between topological spaces, %taking preimages
the preimage map $f^{-1}:\Open(Y)\to \Open(X)$ is a morphism of frames. 
%We equip frames with the cartesian symmetric monoidal structure.  \uli{remove what $\CAlg$}
This refines to a functor
$$\Open(-)^{-1}:\Top^{\op}\to  \Frame\ .$$ 
 Write $\Top_I$ for the non-full wide subcategory of $\Top$ %where
 whose morphisms are open inclusions. For any open inclusion $j:U\to X$, there is a fully faithful functor $j_{\sharp}:\Open(U)\to\Open(X)$ given by $j_{\sharp}(V):=V$. This refines to a functor 
 $$\Open(-)_{\sharp}:\Top_{I}\to  \Poset\ .$$ 

We are mostly interested in frames coming from idempotent coalgebras in presentably symmetric monoidal $\infty$-categories. Let $\cC$ be a %any
 symmetric monoidal $\infty$-category with unit $1$, and let $c:C\rightarrow 1$ be a %any 
 morphism. We say that $(C,c)$ is an idempotent coalgebra if the map $C\otimes c : C\times C\rightarrow C\otimes 1\simeq C$ is an %isomorphism
equivalence in $\cC$. We write $\ICA(\cC)$ for the full subcategory of $\cC_{/1}$ spanned by idempotent coalgebras. The main result of \cite[Sections 3.1 and 3.2]{arXiv:2302.04069} is that, when $\cC$ is presentably symmetric monoidal, $\ICA(\cC)$ is a frame. Moreover, $\ICA$ refines to a functor 
$$\ICA:\CAlg(\Pr^{L}_{\st})\to \Frame\ .$$ %We  
In the following lemma we observe that $\ICA$ %has functoriality
is functorial for functors that are slightly more general than the symmetric monoidal ones.

\begin{lem}\label[lem]{lem:ICAstrongoplax}
	Let $\cC$ and $\cD$ be two presentably symmetric monoidal $\infty$-categories. Suppose that $F:\cC\to\cD$ is a strong oplax symmetric monoidal functor. Then $F$ preserves idempotent coalgebras, and thus induces a functor $\ICA(F):\ICA(\cC)\to\ICA(\cD)$. Moreover, $\ICA$ extends to a functor $$\ICA: \CAlg(\Pr^{L})_{\mathrm{soplax}}\to\Poset.$$
\end{lem}

\begin{proof}
	Let $c:C\to 1$ be an idempotent coalgebra in $\cC$. Let $d:F(C)\to 1$ be the composition of $F(c)$ %and
	with the map $F(1)\to 1$ given by the oplax structure of $F$. We claim that $(F(C),d)$ is an idempotent coalgebra. This follows by %contemplating 
	considering   the diagram
	$$\xymatrix@C=3em@R=3em{
		F(C)\otimes F(C) \ar[r]^{F(C)\otimes F(c)} \ar[d]_{\simeq}
		& F(C)\otimes F(1) \ar[r] \ar[d]_{\simeq}
		& F(C)\otimes 1 \ar[d]^{\simeq} \\
		F(C\otimes C) \ar[r]^{\simeq}_{F(C\otimes c)}
		& F(C\otimes 1) \ar[r]^{\simeq}
		& F(C),
		}$$
	which commutes by definition of an oplax functor. The last claim in the statement of the lemma is %checked 
	easily verified as $\Poset$ is a $1$-category.
\end{proof}

\begin{kor}\label[kor]{kor:ICADfunct}
	Let $D:\LCH^{\op}\to \CAlg(\Pr^{L}_{\st})$ be a functor. %as in \eqref{bsdofjiovsfvsdfv}. 
	The %association 
	assignment $X\mapsto \ICA(D(X))$ extends to a %contravariant
	 functor 
	$$\ICA(D(-))^*:\LCH^{\op}\to\Frame$$ which sends a morphism
	%where any continuous 
	$f$ in $\LCH$ %is mapped 
	to $\ICA(f^*)$. Write $\LCH_I$ for the non-full subcategory of $\LCH$ %where 
	whose  morphisms are open inclusions. If $D$ is additionally 
	localizing, then the association $X\mapsto \ICA(D(X))$ extends to a %covariant 
	functor 
	$$\ICA(D(-))_!:\LCH_I\to\Poset$$ which sends an
	%where any 
	open inclusion $j$ in $\LCH_{I}$ %is mapped 
	to $\ICA(j_!)$.
\end{kor}

\begin{proof}
	The first functoriality follows %comes 
	from the fact that $f^*$ is symmetric monoidal. When $D$ is localizing, the second functoriality is deduced by combining \cref{lem:ICAstrongoplax} with \cref{ggjiwoergergewfrefwref}.  
\end{proof}

Consider a functor  $D:\LCH^{\op}\to \CAlg(\Pr^{L}_{\st})$ %be a functor.Assume that $D$ is as in \eqref{bsdofjiovsfvsdfv} 
and recall the \cref{kohperthertgertg} of a coefficient system.

\begin{lem}\label[lem]{lem:frameopentocidem}
	Let $D$ be a coefficient system. Then, for any $X\in\LCH$, there is a morphism of frames $h_X:\Open(X)\to \ICA(D(X))$ defined by $h_X(U)\coloneqq j_{U\to X,!} p_{U}^{*}1$, where $p_U:U\to\pt$ is the unique map. Moreover, $h$ upgrades to natural transformations $$h:\Open(-)_{\sharp}\to\ICA(D(-))_! \quad \text{and} \quad  h:\Open(-)^{-1}\to\ICA(D(-))^{*}. $$ 
\end{lem}

\begin{proof}
	%For $U$ in $\Open(X)$ let  $j_{U\to X}$ denote the embedding and $p_{U}:U\to *$ denote the projection.  % \uli{make  oservation after definition of localization.Since $j_{U\to X}^{*}$ is  symmetric monoidal, its 
		% left-adjoint $j_{U\to X,!}$    is op-lax symmetric monoidal and therefore preserves  coalgebras.}
	%It actually 
	The functor $p_{U}^{*}$ is symmetric monoidal and $1_{U}\simeq p_{U}^{*}1$ is the tensor unit of $D(U)$, where $1$ denotes the tensor unit of $D(\pt)$.  The tensor unit  in a symmetric monoidal $\infty$-category is clearly an idempotent coalgebra. As observed in \cref{kor:ICADfunct}, the functor $j_{U\to X,!}$ preserves idempotent coalgebras.
	%For every $U$ in $\Open(X)$ we can  thus define the idempotent coalgebra
	%$$h(U):=j_{U\to X,!} p_{U}^{*}1_{U}$$ in $D(X)$.
	If $V$ is a another %second 
	open subset of $X$ with $U\subseteq V$, then we have a map of coalgebras 
	$$ h(U)\simeq  j_{U\to X,!}p_{U}^{*}1\simeq  j_{V\to X,!} j_{U\to V,!}  j_{U\to V}^{*} p_{V}^{*} 1 \to  j_{V\to X,!}   p_{V}^{*}1\simeq h(V)$$ involving the counit of the adjunction $ j_{U\to V,!} \dashv  j_{U\to V}^{*}$.
	Therefore, we find that $h$ %gives 
	yields a well-defined functor $\Open(X)\to \ICA(D(X))$.
	
	We first discuss naturality of $h$. For   open inclusions $V\subseteq U\subseteq X$, we have $ j_{U\to X,!} j_{V\to U,!} 1_{V}\simeq  j_{V\to X,!}  \hat 1_{V}$, which produces the naturality $\Open(-)_{\sharp}\to\ICA(D(-))_!$. 
	If $f:X\to Y$ is a map in $\LCH$,
	then the  square  \begin{equation}\xymatrix{ \Open(Y)\ar[r]^-{h_{Y}}\ar[d]^{f^{-1}} & \ICA(D(Y))\ar[d]^{f^{*}} \\\Open(X) \ar[r]^-{h_{X}} &\ICA(D(X)) } 
	\end{equation} commutes. To this end, we calculate for $U$ in $\Open(Y)$ using I-bc (\cref{okhpertgertgrtegetrhehhrt}.\ref{okhpertgertgrtegetrhehhrt111})  that
	$$h_{X}(f^{-1}(U))\simeq j_{f^{-1}(U)\to X,!} 1_{f^{-1}(U)}\simeq   j_{f^{-1}(U)\to X,!}  f_{|f^{-1}(U)}^{*}1_{U} \simeq  f^{*} j_{U\to X,!}  1_{U} \simeq f^{*}h_{Y}(U).
	$$
	%We thus  get a natural transformation of functors 
	%\begin{equation}\label[equation]{freoijoiwergwregrrwreg} h: \Open(-)\to \ICA(D(-)):\LCH^{\op}\to \Poset\ .\end{equation} 
	
	%We now check that the map $h$ from
	%\eqref{freoijoiwergwregrrwreg}  refines to  a natural transformation of functors 
	%\begin{equation}\label{freoijoiwergwregwreg} h: \Open(-)\to \ICA(D(-)):\LCH^{\op}\to \Frame\ .\end{equation}
	To conclude the proof, %we are only left
	it remains to show that for all $X\in\LCH$, $h_X$ is a map of frames, i.e. it preserves finite meets and all joins.
	First, observe that $h_X(X)=1_X$, and so $h_{X}$ preserves final objects. Let $U,V$ be in $\Open(X)$. Using the  I-bc   and  I-pf (\cref{okhpertgertgrtegetrhehhrt}.\ref{okhpertgertgrtegetrhehhrt1111})  we get
	$$   j_{U\to X,!} 1_{U}\otimes_{X}  j_{V\to X,!} 1_{V}
	\simeq  j_{U\to X,!}   j_{U\to X}^{*} j_{V\to X,!} 1_{V}\simeq j_{U\to X,!}    j_{U\cap V\to U,!}   1_{U\cap V}$$ which 
	shows $h(U)\otimes_{X} h(V)\simeq h(U\cap V)$. Hence $ h_{X}$ preserves binary (and hence finite) meets. 
	To show that $h_{X}$ preserves binary (and hence finite) joins, we must further show that $$ \xymatrix{ h(U\cap V)\ar[r]\ar[d] &h(U) \ar[d] \\h(V) \ar[r] & h(U\cup V)} $$
	is a push-out. Using composability of lower shrieks and the fact that $j_{U\cup V\to X, !}$
	preserves%\footnote{\textcolor{blue}{we only need preserves. We check there is a pushout in $D(U\cup V)$ and then shriek it to $D(X)$}. I would leave everything as I wrote it}\fuli{why? Since it it fully faithful?  This is only for notational convenience. We could stick with $U\cup V$ and avoid this discussion}
	 pushouts, we may assume that $U\cup V=X$. Let $i_{X\setminus V\to X}:X\setminus V\to X$ denote  the inclusion of the closed complement of $V$ into $X$. By localization,  the pair $(j_{V\to X}^*, i_{X\setminus V\to X}^*)$ is jointly conservative. Since both functors preserve pushouts, it suffices to prove the above square becomes a pushout after applying each functor individually.

	By $I$-bc, applying $j_{V\to X}^*$ turns both horizontal arrows into equivalences,
	making the  resulting square  a pushout. Applying $i_{X\setminus V\to X}^*$, the left
	vertical arrow becomes an equivalence (since its domain and codomain both vanish by $I$-bc)
	and the right vertical arrow becomes an equivalence (since $U\setminus V = X\setminus V$).
Thus the resulting square is again a pushout.
	%we see that after the both become s, and thus theBy $I$-bc again, we see that after applying $i_{X\setminus V\to X}^*$ 
	
	% we get a pushout since the left vertical arrow is an equivalence since domain and codomain are both isomorphic to the zero object, and the right vertical arrow becomes an equivalence since $U\setminus V = X\setminus V$.
	%We show that the two vertical arrows induce an equivalence on the cofibers of the two horizontal arrows. By localization and \eqref{bwiojoievewrvfds}, this amounts to proving that the map $$j_{U\to X, !}i_{U\setminus (U\cap V)\to U, *}1_{U\setminus (U\cap V)}\to i_{X\setminus V\to X, *}1_{X\setminus V}$$ 
	%is an equivalence. 
	
	 For a   filtered  family 
	$(U_{i})_{i\in I}$  in $\Open(X)$ with $U=\bigcup_{i\in I}U_{i}$ we have
	$$\colim_{i\in I} h(U_{i})\simeq h(U)$$
	using \eqref{gwerfwrgfdfg}. Hence $h_{X}$ is a frame map, and the proof is finished.
\end{proof}

We are now ready to construct our comparison transformation.

%For the other naturaliy  But and  using the  description \eqref{fwerijfowerfwerfwrf} of $\cB_{X}$
%in the marked equivalence we get 
%$$\cB_{X}\hat j_{U\to X,!}\hat j_{V\to U,!}\hat 1_{V} \simeq \cB_{X}\hat j_{V\to X} \hat 1_{V}\stackrel{!}{\simeq}
%j_{V\to X,!}  \cB_{V} \hat 1_{V}\simeq     j_{U\to X,!} j_{V\to U,!} \cB_{V} \hat 1_{V} \stackrel{!}{\simeq}  j_{U\to X,!} \cB_{U}\hat  j_{V\to U,!} \hat 1_{V} \ . $$

    \begin{prop}\label[prop]{woijgoerferfwrg}
    If $D$ is a coefficient system, then 
    there exists a canonical morphism of
    coefficient systems 
   \begin{equation}\label{boijoidfgbdfgbdgfb}\cB:\Shv(-,D(\pt))\to D(-)\ .
\end{equation} 
    \end{prop}
     \begin{proof}%[Proof of \cref{woijgoerferfwrg}]\phantomsection{}\label{koperghertgertgetrg}
  
% 
% 
%For $\cC$ in $\CAlg(\Pr^{L}_{\st})$ the poset $\ICA(\cC)$ is also a frame 
%
%
%We now check that the map $h$
%Thus \eqref{freoijoiwergwregrrwreg}  refines to  a natural transformation of functors 
% \begin{equation}\label{freoijoiwergwregwreg} h: \Open(-)\to \ICA(D(-)):\LCH^{\op}\to \Frame\ .\end{equation} 
%Furthermore, using the  I-bc   and  I-pf   one checks that 
%$$h(U)\otimes_{X} h(V)\simeq   j_{U\to X,!} 1_{U}\otimes_{X}  j_{V\to X,!} 1_{V}
%\simeq  j_{U\to X,!}   j_{U\to X}^{*} j_{V\to X,!} 1_{V}\simeq j_{U\to X,!}    j_{U\cap V\to U,!}   1_{U\cap V}\simeq
% h(U\cap V)\ .$$
%
%Similarly, the functor $\ICA$ refines to a functor 
%$$\ICA:\CAlg(\Pr^{L}_{\st})\to \CAlg(\Frame)\ .$$
%Thus \eqref{freoijoiwergwregrrwreg}  refines to  a natural transformation of functors 
% \begin{equation}\label{freoijoiwergwregwreg} h: \Open(-)\to \ICA(D(-)):\LCH^{\op}\to \CAlg(\Frame)\ .\end{equation} 

   Recall  the  sheaf-spectrum adjunction \cite[Thm A]{arXiv:2302.04069}:
 %  There exists an adjunction 
    \begin{equation}\label{freoijoiwergwregwreg} \Shv(-,\Sp):\Frame \leftrightarrows   \CAlg(\Pr^{L}_{\st}):\ICA(-)\ .\end{equation} 
    Composing with the adjunction
    \begin{equation*} -\otimes D(\pt):%\Frame 
      \CAlg(\Pr^{L}_{\st}) \leftrightarrows   \CAlg( \Pr^{L}_{\st} )_{D(\pt)/}:\mathrm{fgt}\ \end{equation*}
   yields a new adjunction whose counit is a natural transformation 
    $$ \Shv(\ICA\circ\mathrm{fgt}(-),D(\pt))  \simeq \Shv(\ICA\circ\mathrm{fgt}(-),\Sp )\otimes D(\pt) \to  \id$$
    of endofunctors of  $\CAlg( \Pr^{L}_{\st} )_{D(\pt)/}$. By  precomposing  with $D(-)$,  we get a natural transformation
\begin{equation}\label{hertokopkhpertgertge} \Shv(\ICA\circ \mathrm{fgt}(D(-)) ,D(\pt)) \to D(-)\ .
\end{equation}% \uli{and precompsing the result with $D(-)$}
% we obtain a new one, whose counit gives a natural transformation
Applying $\Shv(-,D(\pt))$ to the map of frames constructed in \cref{lem:frameopentocidem}, we obtain the desired natural transformation $$\cB:\Shv(- ,D(\pt))\to \Shv(\ICA\circ \mathrm{fgt}(D(-)),D(\pt)) \stackrel{\eqref{hertokopkhpertgertge}}{\to} D(-)$$ of functors from $\LCH^{\op}$ to $\CAlg( \Pr^{L}_{\st} )_{D(\pt)/}$.
    
    %We apply the symmetric monoidal functor $\Shv(- ,\Sp)\otimes D(\pt)$ to the transformation \eqref{freoijoiwergwregwreg}, use
    %$$\Shv(X,D(\pt))\simeq \Shv(X,\Sp)\otimes D(\pt)\ ,$$  and compose with the counit of the adjunction \eqref{freoijoiwergwregwreg}  in order to get the desired $D(\pt)$-linear transformation \fuli{improve by working in $\Pr^{L}_{D(\pt)/}$}
    %$$\cB:\Shv(- ,D(\pt))\to \Shv(\ICA(D(-)))\otimes D(\pt)\to D(-)$$ of functors from $\LCH^{\op}$ to $\CAlg( \Pr^{L}_{\st} )_{D(\pt)/}$.

% We used this abstract machine in order to  write down a coherent definition of the transformation $\cB$.
 % If we fix $X$ in $\LCH$, then we 
 %have the diagram of functors with values in    $ \CAlg(\Cat_{\infty})$  \begin{equation}\label{fwerijfowerfwerfwrf}\xymatrix{\Open(X)\ar[r]^-{h}\ar[dr]_{U\mapsto \hat j_{U\to X,!}\hat 1_{U}}&\ICA(D(X))\ar[r]^-{\incl}& D(X)\\&\Shv(X,D(\pt))\ar@{..>}[ur]_{\cB_{X}}&}\ ,
%\end{equation}
%where $\cB_{X}$ is the initial $D(\pt)$-linear functor making this diagram commutative.
% left Kan extension of the upper line along the left diagonal arrow.

%By construction $\cB$  for every $X$ in $\LCH$ the component $\cB_{X}$
%left-adjoint. 
%Hence
%we get\fuli{improve on this point} the desired  natural transformation of $D(\pt)$-linear functors 
%$$\cB:\Shv(-,D(\pt))\to D(-):\LCH^{\op}\to \CAlg(\Pr^{L}_{\st})\ .$$\fuli{For a better argument giving coherence use \cite{arXiv:2302.04069}.}
%In order 
To show that 
$\cB$  is a morphism of coefficient systems, we must show that it is compatible with $I$ (\cref{kohperhretgertge}.\ref{thokerpthertgertge}).
 Let $j_{U\to X}:U\to X$ be an open inclusion.
 We must show that the Beck-Chevalley map 
 $ j_{U\to X,!}\cB_{U}\to \cB_{X}\hat j_{U\to X,!}$ is an equivalence.
 By $D(\pt)$-linearity and since the components of $\cB$ preserve colimits it suffices to show that this is an  equivalence when evaluated on the objects  $  \hat j_{V\to U,!}  \hat 1_{V}$ in $\Shv(U,D(\pt))$ for all $V$ in $\Open(U)$. This follows from the naturality $h$ shown in \cref{lem:frameopentocidem} with respect to the covariant functoriality on open inclusions. \footnote{Since the isomorphism produced in the lemma is in the poset of idempotent coalgebras, it necessarily coincides with the Beck-Chevalley map.}. 
    \end{proof}
  
    \begin{rem} 
    Recall the universal property of cosheaves 
    $$\Fun^{\colim}_{D(\pt)}(\Shv(X,D(\pt)),D(X)) \stackrel{\simeq}{\to} \CoSh(X,D(X))$$
  which is given by %the 
  restriction along the functor $\Open(X)\to \Shv(X,D(\pt)) $, $U\mapsto \hat{j}_{U\to X,!}\hat{1}_{U}$.  By construction, the functor   $\cB_{X}:\Shv(X,D(\pt))\to D(X)$ in \eqref{boijoidfgbdfgbdgfb} is determined by the cosheaf
   $U\mapsto h(U):=j_{U\to X,!}1_{U}$.
%    
%    For $X$ in $\LCH$
%    we consider the functor $\Open(X)\to D(X)$ given by
%    $$U\mapsto j_{U\to X,!}1_{U}\ ,$$
%    where $j_{U\to X}:U\to X$ is the inclusion.
%    Since $D$ satisfies canonical descent, this functor is a $D(X)$-valued cosheaf.
%    The functor $\cB_{X}:\Shv(X,D(\pt))\to D(X)$
%    is the one determined by this cosheaf and the universal property\fuli{reference} of cosheaves
%    $$\Fun^{\colim}(\Shv(X,D(\pt)),D(X))\simeq \CoSh(X,D(\pt))\ .$$  In particular we have equivalences
%    $$\cB_{X}(\hat j_{U\to X,!}     \hat 1_{X} )\simeq j_{U\to X,!}1_{X}\ .$$
 \hB
        \end{rem}

\subsection{Properties of $\cB$}

We assume that   $D:\LCH^{\op}\to \CAlg(\Pr^{L}_{\st})$ is a coefficient system.
In this subsection, we study some properties of $\cB$ from \eqref{boijoidfgbdfgbdgfb} that will be essential in the proof of the main theorem.

\begin{ddd}\label[ddd]{koh0perthertgetge}
For every $X$ in $\LCH$, the right-adjoint   $\cK_{X}$ in
% we let $\cK_{X}$ denote the right-adjoint in 
\begin{equation}\label{her09tuiortherte}\cB_{X}:\Shv(X,D(\pt))\leftrightarrows D(X):\cK_{X}\ .
\end{equation}
is called the associated sheaf functor.
\end{ddd}

\begin{rem} \label[rem]{kothprthertgertge} For $U$ in $\Open(X)$,
 %taking the  
 passing to  right-adjoints in $j_{U\to X,!}\cB_{U}\simeq \cB_{X}\hat{j}_{U\to X,!}$, we get
 \begin{equation}\label{gergwergwerfwrefw444}\cK_{U}j_{U\to X}^{*}\simeq \hat{j}_{U\to X}^{*}\cK_{X}\ .
\end{equation}
 Furthermore, for every map $f:X\to Y$, passing to right-adjoints in % taking the right-adjoint of
 $\cB_{X}\hat{f}^{*}\simeq  f^{*}\cB_{Y}$ we get
 \begin{equation}\label{gergwergwerfwrefw4444}\cK_{Y}f_{*}\simeq \hat{f}_{*}\cK_{X}\ .\end{equation} \hB
\end{rem}

Recall the definition of $\ev_{U}$. from \cref{koptrhrtegtergrtgertg}.
\begin{lem}
For every $U$ in $\Open(X)$, we have the equivalence of functors $$\cK_{X}(-)(U)\simeq \ev_{U}(-):D(X)\to D(\pt)\ .$$
\end{lem}
\begin{proof}
We calculate % and the relations observed in \cref{kothprthertgertge}, that
$$\cK_{X}(-)(U)\stackrel{def}{\simeq} \hat p_{U,*}\hat j_{U\to X}^{*} \cK_{X}(-)\stackrel{\eqref{gergwergwerfwrefw444}}{\simeq} \hat p_{U,*}  \cK_{U}   j_{U\to X}^{*}(-)\stackrel{\eqref{gergwergwerfwrefw4444}}{\simeq}  \cK_{\pt}
p_{U,*}j_{U\to X}^{*}(-)%\uli{\simeq \cK_{\pt} p_{U,*}j_{U\to X}^{*}(-) }
\simeq \ev_{U}(-)\ ,$$ using $\cK_{\pt}\simeq \id$ and \cref{koptrhrtegtergrtgertg} in the final step.
\end{proof}

\begin{kor}\label[kor]{kohperthertgertgrtge} The coefficient system
$D$ is section-determined (\cref{lkpohejztjrtz}) on $X$ if and only if  $\cK_{X}$ is conservative.
\end{kor}

 For the proof of our main theorem,  we must check that
$\cB$ is a morphism of six  functor formalisms  on the standard Nagata context on $\LCH$. To this  end, we must show that for every proper map $f:X\to Y$ in $\LCH$ the square  \begin{equation}\label{guewrguwe9rug9w8erfewrfwer}\xymatrix{ \Shv(Y,D)\ar[r]^{\cB_{Y}}\ar[d]^{\hat f^{*}} &D(Y) \ar[d]^{f^{*}} \\ \Shv(X,D)\ar[r]^{\cB_{X}} & D(X) } 
\end{equation}
is vertically right adjointable. In detail, this means that the canonical Beck-Chevalley map
\begin{equation}\label{gwregerefwrefwer}BC_{f}:\cB_{Y}\hat f_{*}\stackrel{\eta }{\to} f_{*}f^{*}   \cB_{Y}\hat f_{*} \stackrel{\simeq,!}{\to}
	f_{*}   \cB_{X}\hat f^{*}  \hat f_{*} \stackrel{\hat \epsilon}{\to}   f_{*}   \cB_{X} 
\end{equation}    between functors from $\Shv(X,D(\pt))$ to $D(Y)$ is an equivalence. We first observe that \eqref{gwregerefwrefwer} has an alternative description that will be more convenient for our purposes. Since $\cB$ has a right adjoint, one can also consider the map
\begin{equation}\label{eq:BC'}
	BC'_{f}:\cB_{Y}\hat f_{*}\stackrel{\eta^X }{\to} \cB_{Y}\hat f_{*}\cK_X \cB_X \stackrel{\simeq,!}{\to}
	\cB_{Y}\cK_Y f_*\cB_X \stackrel{\epsilon^Y}{\to}   f_{*}   \cB_{X}.
\end{equation}

\begin{lem}\label[lem]{kopgwefref9w}
	The maps \eqref{gwregerefwrefwer} and \eqref{eq:BC'} are equivalent.%homotopic. 
\end{lem}

%\begin{tikzpicture}[scale=0.8, thick]
%  % Die Box für die Transformation alpha
%  \node[draw, rectangle, fill=white, minimum width=2.5cm, minimum height=0.8cm] (alpha) at (0,0) {$\alpha$};
%
%  % L2 geht direkt von unten nach oben in die Box
%  \draw[->] (0.6, -3) node[below] {$L_2$} -- (0.6, -0.4);
%
%  % L3 geht direkt aus der Box nach oben
%  \draw[->] (-0.6, 0.4) -- (-0.6, 3) node[above] {$L_3$};
%
%  % L1 (oben rechts aus der Box) wird durch die Koeinheit \epsilon_1 nach unten zu R1 gebogen (Cap)
%  \draw[->] (0.6, 0.4) .. controls (0.6, 2) and (2.5, 2) .. (2.5, 0.5) -- (2.5, -3) node[below] {$R_1$};
%  \node at (1.5, 1.8) {$\epsilon_1$};
%
%  % L4 (unten links in die Box) kommt von oben (R4) und wird durch die Einheit \eta_4 gebogen (Cup)
%  \draw[->] (-2.5, 3) node[above] {$R_4$} -- (-2.5, -0.5) .. controls (-2.5, -2) and (-0.6, -2) .. (-0.6, -0.4);
%  \node at (-1.5, -1.8) {$\eta_4$};
%\end{tikzpicture}
%
%\begin{equation*}
%\vcenter{\xymatrix@C=4em@R=4em{
%    \cD \ar[r]^{R_1} \ar[d]_{L_3 R_1} & \cC \ar[r]^{L_1} \ar[d]^{L_2} \ar@{=>}[dl]_{\eta_4 L_2 R_1} & \cD \ar[d]^{L_3} \\
%    \cF \ar[r]_{R_4} \ar@{=>}[ur]_{\alpha R_1} & \cE \ar[r]_{L_4} & \cF \ar@{=>}[ul]_{\epsilon_1}
%}}
%\quad = \quad \beta \quad = \quad
%\vcenter{\xymatrix@C=4em@R=4em{
%    \cD \ar[r]^{R_1} \ar[d]_{L_2 R_1} \ar@{=>}[dr]|{\beta} & \cC \ar[d]^{L_2} \\
%    \cE \ar[r]_{R_4} & \cE
%}}
%\end{equation*}

\begin{proof} This is a general $2$-categorical fact for a square of left-adjoint functors.
The Beck-Chevalley map for vertical right-adjointability is equivalent to
the Beck-Chevalley map for horizontal left-adjointability of the square of right-adjoints.
This fact can be verified by 
 standard %but tedious 
computation  using the calculus of mates. For instance, it follows from \cite[Lemma F.10]{Cnossen2024TwistedAmbidexterity} by taking $\alpha = BC_f$.  % \cite{Kelly_1974}.}\fuli{Check this reference. I did the string argument, really.}
%	\textcolor{blue}{really bad to write down... but true. do you know a reference?}
\end{proof}

We first study the case when $f$ is a closed immersion.

\begin{lem}\label[lem]{lem:BClosedimm}
	Let $D$ be a coefficient system, and let $i:Z\to X$ be the inclusion of a closed subset. Then the square 
	\begin{equation}\xymatrix{ \Shv(X,D)\ar[r]^{\cB_{X}}\ar[d]^{\hat i^{*}} &D(X) \ar[d]^{i^{*}} \\ \Shv(Z,D)\ar[r]^{\cB_{Z}} & D(Z) } 
	\end{equation}
	is vertically right adjointable.
\end{lem}

\begin{proof}
	We must show 
	that the map \begin{equation}\label{eq:BCclosedimm}
		\cB_X\hat i_*\to i_*\cB_Z
	\end{equation} is an equivalence. Let $j: X\setminus Z\to X$ be the inclusion of the open complement of $Z$. It suffices to show that \eqref{eq:BCclosedimm} is an equivalence after applying %composing with 
	$i^*$ and $j^*$ individually. This is immediate  using  $f^*\cB_{B}\simeq\cB_{A} \hat f^*$ for all $f:A\to B$ in $\LCH$ and the 
	relations  $j^{*}i_{*}\simeq 0$, $\hat j^{*} \hat i_{*}\simeq 0$, $i^{*}i_{*}\simeq \id$ and $\hat i^{*}\hat i_{*}\simeq \id$
	provided by the localization property for $D$, and for $\Shv(-,D(\pt))$.
	
\end{proof}

We now %want to 
focus on  the special case where $Y=\pt$. Note that $\cB_{\pt}$ is an equivalence. Since $\cB_{X}$ is symmetric monoidal, the Beck-Chevalley map \eqref{gwregerefwrefwer} evaluated at the monoidal unit gives \begin{equation}\label{eq:BConcoh}\Gamma^{\Shv(-,D(\pt))}(X) \to \Gamma^{D}(X).\end{equation} We show that this upgrades to a natural transformation of cohomology theories.

\begin{lem}
	Let $D$ be a %any 
	coefficient system. There is a natural transformation of functors \begin{equation}\label{eq:natonGamma}
		\Gamma^{\cB}:\Gamma^{\Shv(-,D(\pt))} \to \Gamma^{D}(-):\LCH^{\op}\to D(\pt)\ ,
	\end{equation} %of functors $\LCH^{\op}\to D(\pt)$, 
	whose component at any $X$ in $\LCH$ is given by the map \eqref{eq:BConcoh}. % induced by the Beck-Chevalley transformation.
\end{lem}

\begin{proof}
	Let $X$ be any locally compact Hausdorff space, and let $f:X\to\pt$ be the unique map. Since $\cB_X$ admits a right adjoint and the target of $f$ is $\pt$, after precomposing with $\hat f^*$, the map \eqref{eq:BC'} (which is equivalent to \eqref{gwregerefwrefwer} by \cref{kopgwefref9w})  can be described as 
	\begin{equation}\label{eq:betterBC}\hat f_*\hat f^*\to \hat f_* \cK_X\cB_X\hat f^*\simeq f_* f^*\end{equation}
	where the first arrow is the unit of the adjunction $\cB_X\dashv\cK_X$. We explain how to upgrade \eqref{eq:betterBC} to the $X$ component of a natural transformation of functors $$\LCH^{\op}\to\Fun(D(\pt), D(\pt))\ .$$ The desired transformation \eqref{eq:natonGamma} is then obtained by postcomposing with \begin{equation}\label{gwerferfwergwregw45}
\ev_{1}: \Fun(D(\pt), D(\pt))\to D(\pt)\ .\end{equation} 
Let $\Cat^L$ be the wide subcategory of $\Cat$ whose morphisms are left-adjoint functors. Arguing as in \cite[Proposition 3.3]{Volpe2025SixOperations}, one can construct a functor \begin{equation}\label{werfwerfwefwerf}T:\Cat^{L}_{D(\pt)/}\to \Fun(D(\pt), D(\pt))\ .
\end{equation} Its value on an object $f^*:D(\pt)\to\cC$ of $\Cat^{L}_{D(\pt)/}$ is given by $f_*f^*:D(\pt)\to D(\pt)$, and its action on morphisms is given by inserting the units of the respective adjunctions. By \cref{woijgoerferfwrg}, $\cB$ is, in particular, a natural transformation of functors $\LCH^{\op}\to\Cat^L_{D(\pt)/}$. Postcomposing with the functors $T$ from \eqref{werfwerfwefwerf} and $\ev_{1}$ from \eqref{gwerferfwergwregw45} yields \eqref{eq:natonGamma}.
\end{proof}

\begin{prop}\label[prop]{prop:Gamma(B)eqonfinitarycan6ff}
	Let  $D $ be a  six  functor  formalism on the standard Nagata context on $\LCH$,  %which
	that has canonical descent and  whose  associated cohomology  functor $\Gamma^{D}$  is finitary. Assume moreover that $D(\pt)$ is dualizable. Then the transformation \eqref{eq:natonGamma} between the associated cohomology functors is a natural equivalence.
\end{prop}

\begin{proof}
	Since both functors satisfy strong excision (see \cref{kopthrthertgertg}),
	they also satisfy closed excision (see \cref{ghuwegerfwefwerf}). 
	The functor $\Gamma^{D}$ is finitary by assumption, and $\Gamma^{\Shv(-,D(\pt))}$ is  also finitary as observed in \cref{kopherhtgertg}.
	
	%Furthermore, by \cref{hkoptrhertgertger} both functors are finitary.  %They are furthermore continuous, where we refer to 
	%\cite[Thm. 3.6.10]{NKP} for the cohomology associated to sheaves, and to our assumptions for $\Gamma^{D}$.
	Let $\Fun^{\desc}(\CH^{\op},D(\pt))$ denote the full category of $ \Fun(\CH^{\op},D(\pt))$ spanned by finitary\footnote{In \cite{NKP}, this property is called {\em profinite descent}.} functors which satisfy   closed excision.   If $D(\pt)$ is dualizable, then we can apply a   theorem of Clausen \cite[Thm. 3.6.2]{NKP}  stating that the  evaluation at $\pt$
	$$(F\mapsto F(\pt)):\Fun^{\desc}(\CH^{\op},D(\pt))\to D(\pt)$$
	is an equivalence of $\infty$-categories.  So it suffices to show that 
	$\Gamma^{\cB}_{\pt}$ is an equivalence which is obvious. 
\end{proof}

%This finishes the proof in the Case \ref{gjiopwegerfwerfwerf}.

\begin{kor}\label[kor]{jitwoprtherhetrhrh}
Let  $D $ be a  six  functor  formalism on the standard Nagata context on $\LCH$,  that has canonical descent and  whose  associated cohomology  functor $\Gamma^{D}$  is finitary. Assume moreover that $D(\pt)$ is dualizable.   
%and one of the following holds:
 % \begin{enumerate}
 %\item  \label{gjiopwegerfwerfwerf} $D(\pt)$ is dualizable
 %\item  \label{gjiopwegerfwerfwerf1}  $\Gamma^{D}$ is  homotopy invariant.
 %\end{enumerate}
  Then for every compact 
 $X$  in $\LCH$
 the square 
 $$\xymatrix{ \Shv(\pt,D)\ar[r]_{\simeq}^{\cB_{\pt}} \ar[d]^{\hat f^{*}} &D(\pt) \ar[d]^{f^{*}} \\ \Shv(X,D)\ar[r]^{\cB_{X}} & D(X) } $$
 is vertically right adjointable.
\end{kor}
\begin{proof}%[Proof of \cref{jitwoprtherhetrhrh}]
      %Let  $f:X\to *$ be the projection map. We identify $\Shv(*,D(\pt))$ with $D(\pt)$ via $\cB_{\pt}$.
  %We   must show that   the Beck-Chevalley map
    %\begin{equation}\label{gwregerefwrefwer}BC_{X}:\hat f_{*}\stackrel{\epsilon }{\to} f_{*}f^{*}   \hat f_{*} \stackrel{\simeq,!}{\to}
   %f_{*}   \cB_{X}\hat f^{*}  \hat f_{*} \stackrel{\hat \eta}{\to}   f_{*}   \cB_{X} 
%\end{equation}    between functors from $\Shv(X,D(\pt))$ to $D(\pt)$ is an equivalence.
Restriction along the Yoneda embedding
$$\Fun^{\colim}_{D(\pt)}(\Shv(X,D(\pt)),D(\pt))\to \Fun(\Open(X),D(\pt))$$
induces an equivalence onto the full subcategory of $D(\pt)$-valued cosheaves on
$X$. The image of the functor $\hat f_{*}$ under this equivalence is   the cosheaf
$$U\mapsto  \hat f_{*}\hat j_{U\to X,!} \hat 1_{U} \ .$$
Similarly, the  image of the  functor
$ f_{*}   \cB_{X} $ is  the cosheaf
$$U\mapsto   f_{*}  j_{U\to X,!}   1_{U}  \ .$$
We must show that the Beck-Chevalley map 
$$BC_{f}: \cB_{\pt}\hat f_{*}\to f_{*}\cB_{X}$$ from \eqref{gwregerefwrefwer} is an equivalence. This boils down to showing that
$$ BC_{f}: \cB_{\pt}\hat f_{*}\hat j_{U\to X,!} \hat 1_{U}\to f_{*}  \cB_{X} \hat j_{U\to X,!}   \hat 1_{U}$$
is an equivalence for all $U$ in $\Open(X)$.
%\uli{Using $  \cB_{X} \hat j_{U\to X,!}   \hat 1_{U}\simeq    j_{U\to X,!}    1_{U}$ and $\cB_{\pt}\simeq \id$ this map is equivalent 
%to $$ \hat f_{*}\hat j_{U\to X,!} \hat 1_{U}\to  f_{*}  j_{U\to X,!}    1_{U}\ .$$}
Applying $BC_{f}$ %from \eqref{gwregerefwrefwer}  
to the localization sequence
$$\hat j_{U\to X,! } \hat 1_{U}\to  \hat 1_{X}\to \hat  i_{X\setminus U\to X,* }\hat 1_{X\setminus U}$$
we see that it suffices to show that
$$ BC_{f}: \cB_{\pt}\hat f_{*}\hat i_{K\to X,*} \hat 1_{K}\to f_{*}  \cB_{X} \hat i_{K\to X,*}   \hat 1_{K}$$
is an equivalence
% it induces an equivalence on the objects $\hat i_{K\to X,*}\hat 1_{K}$
for all compact subsets $K$ of $X$. 
Abbreviating $i_{*}:=i_{K\to X,*} $ and omitting $\cB_{\pt}\simeq \id$, we claim there is a commutative triangle
\begin{equation}\label{triangle:composeBC}
\small
\xymatrix@C=4em@R=3em{
	\hat f_* \hat i_* \hat 1_K
	\ar[r]^{BC_{f\circ i}}
	\ar[d]_{BC_f}
	&
	f_* i_* 1_K
	\\
	f_* \cB_X \hat i_* \hat 1_K
	\ar[ru]_{f_*(BC_i)}.
	&
}
\end{equation}
By \cref{prop:Gamma(B)eqonfinitarycan6ff} and \cref{lem:BClosedimm} respectively, we know that $BC_{f\circ i}$ and $BC_i$ are equivalences; thus, the commutativity of the triangle  concludes the proof. The commutativity of \eqref{triangle:composeBC} is a standard property of Beck-Chevalley maps; we provide here a full argument for the reader's convenience. Consider the following diagram
\begin{equation}\label{diagram:composeBC}
\small
\xymatrix@C=3.2em@R=3em{
	\hat i_* 
	\ar[d]_{\eta^X}
	\ar[r]^{\eta^K}
	&
	\hat i_* \cK_K\cB_K
	\ar[d]^{\eta^X}
	\ar@/^2.2pc/[rrrd]^{\mathrm{id}}
	&
	&
	&
	\\
	\cK_X\cB_X\hat i_*
	\ar[r]^{\eta^K}
	&
	\cK_X\cB_X \hat i_*\cK_K\cB_K
	\ar[r]^{\simeq}
	&
	\cK_X\cB_X \cK_X i_*\cB_K
	\ar[r]^{\cK_X(\epsilon^X)}
	&
	\cK_X i_* \cB_K
	\ar[r]^{\simeq}
	&
	\hat i_*\cK_K \cB_K.
}
\end{equation}
The left square commutes by naturality of units, and the right triangle commutes by the triangular identities. Therefore  the entire diagram commutes. The commutative triangle \eqref{triangle:composeBC} is then obtained by applying the functor  $\hat f_*$ to \eqref{diagram:composeBC}.
 \end{proof}

 \begin{rem}\label[rem]{rem:htpyinsteadofdual}
 	\cref{prop:Gamma(B)eqonfinitarycan6ff} and \cref{jitwoprtherhetrhrh} remain valid if one replaces the assumption that $D(\pt)$ is dualizable with the assumption that $\Gamma^D$  is homotopy invariant.  We give a brief %very rough 
	sketch of how to prove this claim. A detailed argument will appear 
	%, among other things, 
	in a forthcoming paper of the second named author in collaboration with Maxime Ramzi and Sebastian Wolf. 
 	
 	%Write 
	Let $\Fun^{\desc, \mathrm{htpy}}(\CH^{\op},D(\pt))$ denote %for 
	the full subcategory of $\Fun^{\desc}(\CH^{\op},D(\pt))$ spanned by homotopy invariant functors. Then $$(F\mapsto F(\pt)):\Fun^{\desc, \mathrm{htpy}}(\CH^{\op},D(\pt))\to D(\pt)$$ is an equivalence of $\infty$-categories, even if $D(\pt)$ is not dualizable. The key step %most relevant aspect 
	of the proof is showing that the functor is conservative. Since any compact Hausdorff space is a cofiltered limit of finite polyhedra, one can use finitariness and closed excision to deduce %see 
	that the values of any $F\in\Fun^{\desc, \mathrm{htpy}}(\CH^{\op},D(\pt))$ are determined by its values on $n$-simplices. Homotopy invariance then forces all such values to be equivalent to $F(\pt)$. 
 \end{rem}
 
 %  is used to the effect that equivalences between $D(\pt)$-valued sheaves on finite-dimensional compact Hausdorff spaces can be checked on stalks. At this point, instead of dualizability we could also  use  
 %homotopy invariance of $\Gamma^{\Shv(-,D(\pt))}$ and  $\Gamma^{D}$ finishing the proof in Case \ref{gjiopwegerfwerfwerf1}.

% \begin{kor}
% Under the assumptions that of \cref{jitwoprtherhetrhrh} we have
% $f_{*}\cB_{X}\simeq \hat f_{*}$ for any map $f:X\to\pt$.
% \end{kor}
% \begin{proof}
%We consider the localization sequence for the inclusion $j:X\to X^{+}$ of $X$ into its one-point compactificsation and let $f^{+}:X^{+}\to \pt$ be the corresponding projection.
%We start with the fibre sequence
%$$f^{+}j_{!}j^{*}
% \end{proof}

\subsection{Proofs of the main theorems}

 \begin{theorem}\label[theorem]{hkoperthetrgertge}
Let  $D $ be a  six  functor  formalism on the standard Nagata context on $\LCH$  that %which 
has canonical descent and  whose  associated cohomology  functor $\Gamma^{D}$  is finitary. Assume moreover that $D(\pt)$ is dualizable. 
% \begin{enumerate}
 %\item   $D(\pt)$ is dualizable
 %\item    $H_{D}$ homotopy invariant.
 %\end{enumerate}
 Then the canonical morphism of coefficient systems 
 $$\cB:\Shv(-,D(\pt))\to D(-)$$ from  \eqref{boijoidfgbdfgbdgfb}
 is objectwise fully faithful.
 \end{theorem}
 
\begin{proof}%[Proof of \cref{hkoperthetrgertge}] \phantomsection{}\label{oijiobprbgrbgfbdfgbdf}
We consider the adjunction \eqref{her09tuiortherte}. It suffices to show that the
unit $\id\to \cK_{X}\cB_{X}$ is an equivalence. Since $\Shv(-,D(\pt))$ is section-determined on $X$,
it suffices to show that
\begin{equation}\label{gerwferfwerfw}\hat \ev_{U}\to \hat \ev_{U}\cK_{X}\cB_{X}
\end{equation} is an equivalence for every $U$ in $\Open(X)$. Recall from \cref{koptrhrtegtergrtgertg} that $\hat \ev_{U}\simeq \hat p_{U,*}\hat j_{U\to X}^{*}$.  Using the   relations observed in \cref{kothprthertgertge}, we   calculate
$$  \hat  \ev_{U}\cK_{X}\cB_{X}\simeq \hat p_{U,*}\hat j_{U\to X}^{*} \cK_{X}\cB_{X}\simeq 
 \hat p_{U,*} \cK_{U}  j_{U\to X}^{*} \cB_{X}\simeq     p_{U,*}     j_{U\to X}^{*} \cB_{X} \ .$$
Using an instance of \eqref{gwerfwfeerferrgfdfg},
we can write
$$p_{U,*}     j_{U\to X}^{*}\cB_{X}\simeq \lim_{V\subseteq\subseteq U} p_{V,*}j_{V\to X}^{*}\cB_{X}$$
where the limit runs over open subsets $V$ of $U$ such that $\bar V$ is compact and contained in $U$.
By a cofinality argument, we can extend the index set of the limit to all subspaces whose closure is compact  contained in $U$, and then restrict to all compact subspaces contained in $U$.
This yields
$$p_{U,*}     j_{U\to X}^{*}\cB_{X} \simeq \lim_{K\subset  U} p_{K,*}i_{K\to X}^{*}\cB_{X}\ .$$
Applying \cref{jitwoprtherhetrhrh}, we obtain
$$p_{U,*}     j_{U\to X}^{*}\cB_{X} \simeq \lim_{K\subset  U} p_{K,*}\cB_{K} \hat i_{K\to X}^{*} \simeq
\lim_{K\subset  U} \hat  p_{K,*}  \hat i_{K\to X}^{*} \simeq \hat \ev_{U}\ ,$$
which shows that \eqref{gerwferfwerfw} is an equivalence and completes the proof.
\end{proof}

\begin{kor}\label[kor]{ihoptrhertgertge}
Assume that $D$ satisfies the assumptions stated in \cref{hkoperthetrgertge}.
If $\cK_{X}$ is conservative, then $\cB_{X}$ is an equivalence.  \end{kor}

 %The following is a more precise reformulation of \cref{kohperhertgertgertg} from the introduction. In comparison with  \cref{hkoperthetrgertge} 
% the only additional assumption is section-determinedness.

  \begin{theorem} \label[theorem]{thkoperthertegrtger} Let  $D $ be a  six  functor  formalism on the standard Nagata context on $\LCH$  that %which
  has canonical descent, is section-determined, and  whose  associated cohomology  functor $\Gamma^{D}$  is finitary. Assume moreover that $D(\pt)$ is dualizable.
% \begin{enumerate}
 %\item  $D(\pt)$ is dualizable.
 %\item $\Gamma^{D}$ is homotopy invariant.\end{enumerate}
 Then the transformation
 $$\cB:\Shv(-,D(\pt))\to D(-)$$ from \eqref{boijoidfgbdfgbdgfb}
 is an equivalence of six functor formalisms. 
  \end{theorem}
  \begin{proof}
  It suffices to show that for every $X$ in $\LCH$ the transformation $\cB_{X}$ in \eqref{boijoidfgbdfgbdgfb}
 is an equivalence. In view of   \cref{ihoptrhertgertge},
 it suffices to show that $\cK_{X}$ is conservative.  By \cref{kohperthertgertgrtge} this is equivalent to the assumption that $D$ is section-determined.
  \end{proof}

%\begin{theorem}\label{thkoperthertegrtger}
%Assume that $D$ is a  6 functor calculus on $(\LCH,I,P)$ satisfying canonical descent which is continuous and hyperlocal and such that one of the following holds:\fuli{The question is whether we can use \cref{hkoptrhertgertger} to deduce profinite cohomology descent from continuity }
%\begin{enumerate}
% \item $D$ has profinite cohomology descent and $D(\pt)$ is dualizable
% \item $D$  has profinite cohomology descent and is cohomologically homotopy invariant.
% \end{enumerate}
% Then the transformation in \eqref{boijoidfgbdfgbdgfb}
% is an equivalence of 6 functor calculi. 
%\end{theorem}
%
  \begin{rem}
The statement of \eqref{thkoperthertegrtger} means that
\eqref{boijoidfgbdfgbdgfb} is a morphism of six functor formalisms, i.e., is also compatible with $P$, and that
 the functor
$\cB_{X}:\Shv(X,D(\pt))\to D(X)$ is an equivalence for every  $X$ in $\LCH$.
\end{rem}

\begin{rem}\label[rem]{rem:mainthmhtpyinv}
	By \cref{rem:htpyinsteadofdual}, \cref{thkoperthertegrtger} is still true if we replace the assumption that $D(\pt)$ is dualizable with $\Gamma^D$ being homotopy invariant.
\end{rem}

\begin{kor}
	Any continuous six functor formalism $D$ in the sense of \cite{arXiv:2507.13537} is equivalent to $\Shv(-,D(\pt))$.
\end{kor}

\begin{proof}
	By definition, $D$ is continuous in the sense of \cite{arXiv:2507.13537}  if it satisfies canonical descent, %it 
	is finitary, is stalk-determined, and if $D(\pt)$ is dualizable. By \cref{okhpetrhrferferferfetgerge}, it follows that $D$ is section-determined. Therefore the claim is a direct %desired conclusion is a 
consequence of \cref{thkoperthertegrtger}.
\end{proof}

    \bibliographystyle{alpha}
\bibliography{forschung2021-1}

\end{document}